\documentclass[preprint,12pt]{elsarticle}    

\usepackage{graphicx,epsfig}
\usepackage{enumerate}
\usepackage{amssymb}
\usepackage{amsthm}
\usepackage{amsmath}
\usepackage{centernot}
\usepackage{xcolor}
\usepackage{times}
\usepackage{mathrsfs}
\usepackage[utf8]{inputenc}
\usepackage{subfigure}
\usepackage{multirow}
\usepackage{soul}

\usepackage{amsmath}
\usepackage{mathtools}
\usepackage{appendix}

\usepackage{wrapfig}
\usepackage{setspace}
\usepackage{booktabs}
\usepackage{lscape}
\usepackage{algorithm,algorithmicx}
\usepackage{algcompatible}
\usepackage{bm}
\usepackage{graphicx,epsfig}
\usepackage{xfrac}
\usepackage{url}
\usepackage{multirow}
\usepackage{longtable}
\usepackage[linkcolor=blue, urlcolor=blue, citecolor=blue,
colorlinks, bookmarks]{hyperref}

\newtheorem{thm}{Theorem}[section]
\theoremstyle{definition}
\newtheorem{dfn}{Definition}[section]

\newtheorem{lem}{Lemma}[section]
\newtheorem{cor}{Corollary}[section]
\newtheorem{example}{Example}[section]
\newtheorem{rmrk}{Remark}[]

\DeclareMathOperator*{\argeff}{argeff}

\makeatletter
\newcommand{\thickhline}{%
    \noalign {\ifnum 0=`}\fi \hrule height 1pt
    \futurelet \reserved@a \@xhline
}

\journal{Information Sciences}

\begin{document}

\begin{frontmatter}

\title{\textbf{Generalized-Hukuhara-Gradient Efficient-Direction Method to Solve Optimization Problems with Interval-valued Functions and its Application in Least Squares Problems}}
\author[iitbhu_math]{Debdas Ghosh\corref{cor1}}
\ead{debdas.mat@iitbhu.ac.in}
\author[iitbhu_math]{Amit Kumar Debnath}
\ead{amitkdebnath.rs.mat18@itbhu.ac.in}
\author[iitbhu_math]{Ram Surat Chauhan}
\ead{rschauhan.rs.mat16@itbhu.ac.in}
\author[t_mx]{Oscar Castillo}
\ead{ocastillo@tectijuana.mx}
\address[iitbhu_math]{Department of Mathematical Sciences,  Indian Institute of Technology (BHU) Varanasi \\ Uttar Pradesh--221005, India}
\address[t_mx]{Tijuana Institute of Technology, Tomas Aquino, Tijuana 22414, Mexico}
\cortext[cor1]{Corresponding author}

\begin{abstract}

\noindent \hl{This article proposes a general \textit{$gH$-gradient efficient-direction method} and a \textit{$\mathcal{W}$-$gH$-gradient efficient method} for the optimization problems with interval-valued functions}. The convergence analysis and the step-wise algorithms of both the methods are presented. It is observed that the $\mathcal{W}$-$gH$-gradient efficient method converges linearly for a strongly convex interval-valued objective function. To develop the proposed methods and to study their convergence, the idea of strong convexity and sequential criteria for $gH$-continuity of interval-valued function are illustrated. In the sequel, a new definition of $gH$-differentiability for interval-valued functions is also proposed. The new definition of $gH$-differentiability is described with the help of a newly defined concept of linear interval-valued function. It is noticed that the proposed $gH$-differentiability is superior to the existing ones. For a $gH$-differentiable interval-valued function, the relation of convexity with the $gH$-gradient of an interval-valued function and an optimality condition of an interval optimization problem are derived. For the derived optimality condition, a notion of \textit{efficient direction} for interval-valued functions is introduced. The idea of efficient direction is used to develop the proposed gradient methods. As an application of the proposed methods, the least square problem for interval-valued data by $\mathcal{W}$-$gH$-gradient efficient method is solved. The proposed method for least square problems is illustrated by a polynomial fitting and a logistic curve fitting. \\
\end{abstract}

\begin{keyword}
Interval-valued functions\sep Convexity\sep Strong convexity\sep $gH$-continuity\sep  $gH$-gradient\sep $gH$-differentiability\sep Efficient solution\sep Efficient direction\sep Least square problems.
\\ \vspace{0.7cm}
AMS Mathematics Subject Classification (2010): 90C30 $\cdot$ 65K05
\end{keyword}

\end{frontmatter}

\section{Introduction}

\hl{Each area of science, engineering, management, economics, and other practices, uses optimization techniques extensively}. Optimization techniques assist us to find the best under specified circumstances. The optimization problems with interval-valued functions (IVFs), known as interval optimization problems (IOPs), has become a significant research topic over the last two decades due to inherent imprecise and uncertain events in different real-world events.
In this paper, we attempt to derive a technique for IOPs to capture its solution set. The proposed method reduces to the steepest descent method for the optimization problems with real-valued functions.

\subsection{Literature Survey}

\hl{The ordering and subtraction of intervals has always been a issue in pursuit of an optimal solution for IOPs} \cite{ghosh2017spc}. In order to deal with interval-valued data, Moore introduced interval arithmetic \cite{moore1966,moore1987}. However, with the interval arithmetic in \cite{moore1966,moore1987}, one cannot find the additive inverse of a nondegenerate interval (whose lower and upper limits are different), i.e., for a nondegenerate interval $\textbf{A}$, there does not exist an interval $\textbf{B}$ such that $\textbf{A}\oplus\textbf{B}=\textbf{0}$. Due to this  reason, Wu \cite{wu2007,wu2008,wu2009} used a new concept of difference of intervals, known as Hukuhara difference \cite{hukuhara1967} for the difference of two nonempty, closed, bounded and convex subsets of a real linear space. In spite of the fact that the Hukuhara difference for intervals satisfies $\textbf{A}\ominus_{H}\textbf{A}=\textbf{0}$, $\textbf{A}\ominus_{H}\textbf{B}$ can be calculated only when the width of $\textbf{A}$ is greater than equal to that of $\textbf{B}$. {\color{red} In order to overcome this inefficiency of Hukuhara difference of intervals, the `nonstandard subtraction', introduced by Markov \cite{markov1979}, has been used and named as generalized Hukuhara difference ($gH$-difference) by Stefanini \cite{stefanini2009,stefanini2010}}. The generalized Hukuhara difference can be calculated for any pair of intervals and has the property that  $\textbf{A}\ominus_{gH}\textbf{A}=\textbf{0}$ \cite{stefanini2009}.
\\

In the ordering of intervals, as intervals are not linearly ordered in contrast to the real numbers, Ishibuchi and Tanaka \cite{ishibuchi1990} showed various partial ordering structures and solution concepts for IOPs. They suggested a method to solve a linear IOP by converting it to a bi-objective optimization problem, which is generalized by Chanas and Kuchuta \cite{chanas1996}. For nonlinear IOPs, Ghosh studied a Newton method \cite{ghosh2016newton} and a quasi-Newton method \cite{ghosh2017quasinewton}. Interestingly, many researchers proposed different {\color{red}types} of algorithms to solve various types of practical IOPs, for instance, see \cite{chen2004,chen2004freq,chinneck2000,csendes2001,limbourg2005,wolfe2000,wu2006}. Recently, Ghosh et al. \cite{ghosh2020ordering} introduced variable ordering relations of intervals and proposed an algorithm to obtain the solutions to IOPs. However, research into the applicability of conventional optimization techniques for IOPs is still not concentrated. More surprisingly, although the interplay between geometry and calculus yield optimization techniques, the calculus for IVFs is not rigorously developed until now. \\

In the year of 2007, with the help of a Hausdorff metric between any two intervals, Wu \cite{wu2007} illustrated the concept of continuity of an IVF. In the same article \cite{wu2007}, based on the Hukuhara difference, the concept of Hukuhara-differentiability ($H$-differentiability) of an IVF has been proposed. Accordingly, the KKT optimality conditions for IOPs have been given in \cite{wu2007}. Further, applying the concept of $H$-differentiability, Wu  \cite{wu2007,wu2008,wu2009} studied various duality theories of IOPs. Thereafter, showing the restrictiveness of $H$-differentiability, Chalco-Cano et al.  \cite{chalco2013calculus} developed the calculus of IVFs based on the modified concept of the $gH$-difference, known as generalized-Hukuhara differentiability ($gH$-differentiability). Chalco-Cano et al. \cite{chalco2013kkt} and Ghosh et al. \cite{ghosh2019extended} also derived the KKT conditions and duality theories in the view of $gH$-differentiability. \\

{\color{red} \hl{In the development of interval calculus, calculus for fuzzy-valued functions plays an important role because intervals are particular fuzzy numbers with a special membership function}. In connection with fuzzy calculus, Bede and Gal \cite{bede2005generalizations} introduced generalized (Hukuhara-based) differentiability;  the  paper  motivated  the  search  for  a $gH$-difference for intervals and fuzzy numbers (see \cite{stefanini2009,stefanini2009gh,stefanini2010}) and applications to fuzzy generalized Hukuhara differentiability (see \cite{bede2013generalized}). A recent contribution in this direction is the article by  Stefanini and Arana-Jim{\'e}nez \cite{stefanini2019} which contains definitions of total, directional and partial $gH$-derivatives for multi-variable interval- and fuzzy-valued functions.} \\ 

In the existing literature on interval calculus,  unlike the definition of differentiability of real-valued functions, none of the existing approaches used the concept of a \emph{linear IVF} to define the differentiability of an IVF.  Although similar to the definition of differentiability of real-valued functions, the authors of \cite{ghosh2016newton} and \cite{stefanini2019} introduced the new definitions of $gH$-differentiability for IVFs and studied the properties $gH$-differentiable IVFs. However, none of them also mentioned about the linear IVF and used the concept of linear IVF to define $gH$-differentiability for IVFs. \\

Since the last two decades, with the development of the calculus of IVFs and theories related to IOPs, many techniques, and their algorithmic implementations to obtain the efficient solutions of various types of practical IOPs have been appeared, for instance, see almost all the papers in the references.
However, the majority of the methods are provided from the perspective of conventional bi-objective optimization. Thus, to apply those techniques one has to explicitly express an IVF $\textbf{F}$ in terms of its real-valued lower $\underline{f}$ and upper $\overline{f}$ boundary functions, which is quite restrictive. For example, in a general least square problem for interval-valued data (see Section \ref{sa}), one cannot easily express the interval-valued error function in terms of its lower and upper boundary functions. The authors of  \cite{ghosh2016newton,ghosh2017spc} have studied a parametric form of IOPs and developed the theories and techniques to find efficient solutions to the IOPs with the objective functions that can be parametrically presented. However, for the parametric representation of an IVF one needs its explicit form which is often practically not possible, for instance, consider the function ${\color{red}\boldsymbol{\mathcal{E}} (\beta)}$ in  (\ref{eef}).

\subsection{Motivation and Contribution of the Paper}

\hl{The literature on IOPs shows that there is still no emphasis on the study of conventional optimization strategies for IOPs}. Surprisingly, the basic descent method is not yet developed for IOPs. Further, to derive a technique for IOPs which is similar to the standard descent method, we need to rigorously establish the notion of $gH$-differentiability concept for IVFs. More importantly, it must be kept in mind that the derived technique must be applicable to general IVFs regardless of whether or not
\begin{enumerate}[(i)]
\item the objective function can be expressed parametrically, or
\item the explicit form of the lower and upper function of the objective function can be found.
\end{enumerate}

\hl{After illustrating the concept of a linear IVF, this paper proposes a new definition of $gH$- differentiability}. It is shown that if an IVF is $gH$-differentiable at a point, its $gH$-gradient exists at that point. It is shown that the proposed definition of $gH$-differentiability is superior to the existing ones (see Remark \ref{nd1} for details). With the help of $gH$-gradient, a few characterization results for a $gH$-differentiable convex IVF are derived. Also, several results related to the $gH$-gradient of a strong convex $gH$-differentiable IVF are studied. \\

\hl{Further, with the help of the proposed $gH$-differentiability for IVFs, this article develops a gradient descent method for interval optimization, namely a general $gH$-gradient efficient-direction method for IOPs}. Similar to the steepest descent method, a method is also proposed, named $\mathcal{W}$-gradient efficient method, to obtain efficient solutions of IOPs. The main advantages of the proposed methods are that one needs neither the explicit forms of upper and lower functions of the objective function nor parametric forms of the corresponding IVFs of an IOP. It is shown that the $\mathcal{W}$-gradient efficient method for IOP converges linearly in the case of strong convexity of the interval-valued objective function. In order to develop these methods, the notion of efficient-direction for an IVF and its several characteristics are studied.   \\

\subsection{Delineation}
\hl{The presentation sequence of the proposed work is the following}. The next section covers some basic terminologies and notions of intervals analysis followed by the convexity and a few topics of differential calculus of IVFs. Also, the sequential criteria of $gH$-continuity of an IVF is discussed in Section \ref{sect2}. The concept of a linear IVF, a new concept of $gH$-differentiability of an IVF, and a few characterizations of a $gH$-differentiable convex IVF are given in Section \ref{sdivf}. The concept of efficient solutions and an optimality condition of an IOP are discussed in Section \ref{siopes}. In Section \ref{sgdmiop}, a general $gH$-gradient efficient-direction method for IOPs and a  $\mathcal{W}$-gradient efficient method for IOP are proposed. Their algorithmic implementations and the convergence analysis are also studied in Section \ref{sgdmiop}. The section \ref{sa} deals with the application of $\mathcal{W}$-gradient efficient method for IOPs in least square problems with interval data. Finally, in Section \ref{scf}, a few future directions of this study are given. \\

\section{\textbf{Preliminaries and Terminologies}}\label{sect2}

\noindent \hl{This section provides some basic terminologies and notions on intervals followed by the convexity and a few topics of differential calculus of IVFs.}

\subsection{Arithmetic of Intervals and their Dominance Relation}\label{ssai}

\noindent At first, this section describes the generalized concept of the difference of two intervals and the ordering concepts of intervals. \textcolor{red}{Along with these definitions, we use Moore's interval addition ($\oplus$) multiplication ($\odot$) and division ($\oslash$) \cite{moore1966,moore1987} throughout the paper.} \\

Let the set of real numbers be denoted by $\mathbb{R}$ and the set of all closed and bounded intervals be denoted by $I(\mathbb{R})$. Throughout the article, the elements of $I(\mathbb{R})$ are represented by bold capital letters ${\textbf A}, {\textbf B}, {\textbf C}, \ldots $. To represent an element $\textbf{A}\in I(\mathbb{R})$ in the interval form, the corresponding small letter is used in the following way: $\textbf{A} = [\underline{a}, \overline{a}.$ If $\underline{a}=\overline{a}]$, then $\textbf{A}$ is called a \emph{degenerate interval}.\\

It is to be mentioned that any singleton $\{p\}$ of $\mathbb{R}$ can be represented by an interval $\textbf{P}=[\underline{p},\;\overline{p}]$, where $\underline{p}=p=\overline{p}$. In particular,
\[
\textbf{0}=\{0\}=[0, 0]~~\text{and}~~\textbf{1}=\{1\}=[1, 1].
\]
\begin{rmrk}\label{ria1}
It is easy to check that the addition and multiplication of intervals are commutative, the addition of intervals is associative, and
\[
\textbf{A} \ominus \textbf{B}=\textbf{A} \oplus (-1)\odot\textbf{B}.
\]
\end{rmrk}
Since the property of subtraction of intervals cannot provide an additive inverse of a nondegenerate interval, in this article, we use the $gH$-difference of intervals, which is defined as follows.
\begin{dfn}
(\emph{$gH$-difference of intervals} \cite{stefanini2010}). Let
$\textbf{A}$ and $\textbf{B}$ be two elements of $I(\mathbb{R})$. The $gH$-difference
between $\textbf{A}$ and $\textbf{B}$, denoted $\textbf{A} \ominus_{gH} \textbf{B}
$, is defined by an interval $\textbf{C}$ such that
\[
\textbf{A} =  \textbf{B} \oplus  \textbf{C} ~\text{ or }~ \textbf{B} = \textbf{A}
\ominus \textbf{C}.
\]
It is to be noted that for $\textbf{A} = \left[\underline{a}, \overline{a}\right]$ and
$\textbf{B} = \left[\underline{b}, \overline{b}\right]$,
\[
\textbf{A} \ominus_{gH} \textbf{B} = \left[\min\{\underline{a}-\underline{b},
\overline{a} - \overline{b}\},  \max\{\underline{a}-\underline{b}, \overline{a} -
\overline{b}\}\right].
\]
Thus,
\[
\textbf{A} \ominus_{gH} \textbf{A} = \textbf{0} ~~\text{and}~~ \textbf{0} \ominus_{gH} \textbf{A}=(-1) \odot \textbf{A}.
\]
\end{dfn}
\begin{dfn}
(\emph{Algebraic operations on $I(\mathbb{R})^n$}). Let
$\bar{\textbf{A}} = \left(\textbf{A}_1, \textbf{A}_2, \ldots, \textbf{A}_n\right)^T$ and $\bar{\textbf{B}} =(\textbf{B}_1,$ $\textbf{B}_2, \ldots, \textbf{B}_n)^T$ be two elements of $I(\mathbb{R})^n$. An algebraic operation `$\star$' between $\bar{\textbf{A}}$ and $\bar{\textbf{B}}$, denoted $\bar{\textbf{A}} \star \bar{\textbf{B}}$, is defined by
\[
\bar{\textbf{A}} \star \bar{\textbf{B}}=\left(\textbf{A}_1\star \textbf{B}_1, \textbf{A}_2\star \textbf{B}_2, \ldots, \textbf{A}_n\star \textbf{B}_n\right)^T,
\]
where $\star\in\{\oplus,\ \ominus, \ \ominus_{gH}\}$.
\end{dfn}
\begin{dfn}
(\emph{Dominance relation of interval} \cite{wu2007}). For any two intervals $\textbf{A}$ and $\textbf{B}$ in $I(\mathbb{R})$,
\begin{enumerate}[(i)]
\item if $\underline{a}~\leq~ \underline{b}$ and $\overline{a}~\leq~\overline{b}$, then $\textbf{B}$ is said to be dominated by $\textbf{A}$ and denoted by $\textbf{A}~\preceq~ \textbf{B}$;
\item if either $\underline{a} ~\leq~ \underline{b}$  and $\overline{a} ~<~ \overline{b}$ or $\underline{a} ~<~ \underline{b}$  and $\overline{a} ~\leq~ \overline{b}$ hold, then $\textbf{B}$ is said to be strictly dominated by $\textbf{A}$ and denoted by $\textbf{A}~\prec~ \textbf{B}$;
\item if $\textbf{B}$ is not dominated by $\textbf{A}$, then $\textbf{A}~\npreceq~ \textbf{B}$ and if $\textbf{B}$ is not strictly dominated by $\textbf{A}$, then $\textbf{A}~\nprec~ \textbf{B}$;
\item if $\textbf{A}~\npreceq~ \textbf{B}$ and $\textbf{B}~\npreceq~ \textbf{A}$, then it will be said that none of $\textbf{A}$ and $\textbf{B}$ dominates the other, or $\textbf{A}$ and $\textbf{B}$ are not comparable.
\end{enumerate}
\end{dfn}
One can note that Wu \cite{wu2007,wu2009} used the term `superior than' to describe the dominance relation between two intervals. However, in this article, we use the term `dominated by' instead of `superior than'.
\begin{lem}\label{ldr1} For two elements $\textbf{A}$ and $\textbf{B}$ of $I(\mathbb{R})$,
\begin{enumerate}[(i)]
\item \label{part1} $\textbf{A}~\preceq~ \textbf{B} \Longleftrightarrow \textbf{A}\ominus_{gH}\textbf{B}~\preceq~  \textbf{0}$~\text{ and }
\item \label{part2} $\textbf{A}~\nprec~ \textbf{B} \Longleftrightarrow \textbf{A}\ominus_{gH}\textbf{B}~\nprec~ \textbf{0}.$
\end{enumerate}
\end{lem}
\begin{proof}
See \ref{apdr}.
\end{proof}
\begin{lem}\label{ldr2} For an $\textbf{A} \in I(\mathbb{R})$,
\begin{enumerate}[(i)]
\item  $\textbf{0}~\preceq~ \textbf{A} \Longleftrightarrow (-1)\odot\textbf{A}~\preceq~ \textbf{0}$ and
\item  $\textbf{0}~\nprec~ \textbf{A} \Longleftrightarrow (-1)\odot\textbf{A}~\nprec~ \textbf{0}$.
\end{enumerate}
\end{lem}
\begin{proof}
As $\textbf{0}\ominus_{gH}\textbf{A}=(-1)\odot\textbf{A}$, replacing $\textbf{A}$ by $\textbf{0}$ and $\textbf{B}$ by $\textbf{A}$ in Lemma \ref{ldr1}, we get the required results.
\end{proof}
\begin{dfn}
(\emph{Norm on $I(\mathbb{R})$} \cite{moore1966}). For an $\textbf{A} = \left[\underline{a}, \bar{a}\right]$ in $ I(\mathbb{R})$, the function ${\lVert . \rVert}_{I(\mathbb{R})} : I(\mathbb{R}) \rightarrow \mathbb{R}^+$, defined by
\[
{\lVert \textbf{A} \rVert}_{I(\mathbb{R})} = \max \{|\underline{a}|, |\bar{a}|\},
\]
is a norm on $I(\mathbb{R})$.
\end{dfn}
\begin{dfn}
(\emph{Norm on $I(\mathbb{R})^n$} \cite{moore1987}). For an $\bar{\textbf{A}} = \left(\textbf{A}_1, \textbf{A}_2, \ldots, \textbf{A}_n\right)^T$ in $ I(\mathbb{R})^n$, the function ${\lVert \cdot \rVert}_{I(\mathbb{R})^n} : I(\mathbb{R}) \rightarrow \mathbb{R}^+$, defined by
\[
{\lVert \bar{\textbf{A}} \rVert}_{I(\mathbb{R})^n} = \sum_{i=1}^n {\lVert \textbf{A}_i \rVert}_{I(\mathbb{R})}
\]
is a norm on $I(\mathbb{R})^n$.
\end{dfn}

In this article, although we use the notions `${\lVert \cdot \rVert}_{I(\mathbb{R})}$' and `${\lVert \cdot \rVert}_{I(\mathbb{R})^n}$' to denote the norms on $I(\mathbb{R})$ and $I(\mathbb{R})^n$, respectively, we simply use the notion `${\lVert \cdot \rVert}$' to denote the usual Euclidean norm on $\mathbb{R}^n$.


\subsection{Convexity and Basic Differential Calculus of Interval-valued Functions}
\noindent Let $\mathcal{X}$ be a nonempty subset of $\mathbb{R}^n$. An IVF $\textbf{F}:\mathcal{X} \rightarrow I(\mathbb{R})$, for each argument point $x \in \mathcal{X}$, is presented by the ontic (see \cite{couso2014statistical}) way: 
\[
\textbf{F}(x)=\left[\underline{f}(x),\
\overline{f}(x)\right],
\]
where $\underline{f}$ and $\overline{f}$ are real-valued functions on $\mathcal{X}$. The functions $\underline{f}$ and $\overline{f}$ are called the lower and the upper functions of $\textbf{F}$, respectively. 
\begin{dfn}
(\emph{Convex IVF} \cite{wu2007}). Let $\mathcal{X} \subseteq
\mathbb{R}^n$ be a convex set. An IVF
$\textbf{F}: \mathcal{X} \rightarrow I(\mathbb{R})$
is said to be a convex function if for any two vectors $x_1$ and $x_2$ in
$\mathcal{X}$,
\[
\textbf{F}(\lambda_1 x_1+\lambda_2 x_2)~\preceq~ \lambda_1\odot\textbf{F}(x_1)\oplus\lambda_2\odot\textbf{F}(x_2)
\]
for all $\lambda_1,~\lambda_2\in[0,\ 1]$ with $\lambda_1+\lambda_2=1$.
\end{dfn}
\begin{rmrk}\label{rc1}(See \cite{wu2007}).
$\textbf{F}$ is convex if and only if $\underline{f}$
and $\overline{f}$ are convex.
\end{rmrk}
\begin{dfn}\label{strongly_convex_IVF}	(\emph{Strongly convex IVF}). Let $\mathcal{X}$ be a nonempty convex subset of $\mathbb{R}^n$. An IVF $\textbf{F}:\mathcal{X} \rightarrow I(\mathbb{R})$ is said to be strongly convex on $\mathcal{X}$ if there exists a convex IVF $\textbf{G}:\mathcal{X} \rightarrow I(\mathbb{R})$ and a $\sigma>0$ such that
\[
\textbf{F}(x) = \textbf{G}(x)\oplus \frac{1}{2}\rVert x \rVert^2\odot [\sigma, \sigma] \text{ for all } x \in \mathcal{X}.
\]
\end{dfn}
\begin{rmrk}\label{rscv1}
It is to be observed that \[
\textbf{F}(x) = \textbf{G}(x)\oplus \frac{1}{2}\rVert x \rVert^2\odot [\sigma, \sigma]
\]
implies
\begingroup\allowdisplaybreaks\begin{align*}
& \left[\underline{g}(x), \overline{g}(x)\right]\oplus \frac{1}{2}\rVert x \rVert^2\odot [\sigma, \sigma]=\left[\underline{f}(x), \overline{f}(x)\right]\\
\text{or},&\left[\underline{g}(x)+\frac{\sigma}{2}\rVert x \rVert^2, \overline{g}(x)+\frac{\sigma}{2}\rVert x \rVert^2 \right]=\left[\underline{f}(x), \overline{f}(x)\right]\\
\text{or},&~\underline{g}(x)=\underline{f}(x)-\frac{\sigma}{2}\rVert x \rVert^2~\text{and}~\overline{g}(x)=\overline{f}(x)-\frac{\sigma}{2}\rVert x \rVert^2.
\end{align*}\endgroup
Therefore,
\begingroup\allowdisplaybreaks\begin{align*}
\textbf{F}~\text{is strongly convex}&\Longleftrightarrow\textbf{G}~\text{is convex}\\
&\Longleftrightarrow\underline{g}~\text{and}~\overline{g}~\text{are convex, by Remark \ref{rc1}}\\
&\Longleftrightarrow\underline{f}~\text{and}~\overline{f}~\text{are strongly convex}.
\end{align*}\endgroup
\end{rmrk}
\begin{dfn}
(\emph{$gH$-continuity} \cite{ghosh2016newton}).
 Let $\textbf{F}$ be an IVF on a nonempty subset $\mathcal{X}$ of $\mathbb{R}^n$. Let $\bar{x}$ be an interior point of $\mathcal{X}$
 and $d\in\mathbb{R}^n$ be such that $\bar{x}+d\in\mathcal{X}$. The function
 $\textbf{F}$ is said to be continuous at $\bar{x}$ if
\[
\lim_{\lVert d \rVert\rightarrow
0}\left(\textbf{F}(\bar{x}+d)\ominus_{gH}\textbf{F}(\bar{x})\right)=\textbf{0}.
\]
\end{dfn}
\begin{lem}\label{lc1}
An IVF $\textbf{F}$ on a nonempty subset $\mathcal{X}$ of $\mathbb{R}^n$ is $gH$-continuous if and only if
$\underline{f}$ and $\overline{f}$ are continuous.
\end{lem}
\begin{proof}
See \ref{aplc1}.
\end{proof}
\begin{lem}(\emph{Sequential criteria of $gH$-continuity}).\label{lsc}
An IVF $\textbf{F}$ on a nonempty subset $\mathcal{X}$ of $\mathbb{R}^n$ is $gH$-continuous at a point $\bar{x}\in\mathcal{X}$ if and only if for every sequence $\{x_n\}$ in $\mathcal{X}$ converging to $\bar{x}$, the sequence $\{\textbf{F}(x_n)\}$ converges to $\textbf{F}(\bar{x})$.
\end{lem}
\begin{proof}
See \ref{aplsc}.
\end{proof}
\begin{dfn}(\emph{$gH$-Lipschitz continuous IVF} \cite{ghosh2019derivative}). Let $\mathcal{X}\subseteq \mathbb{R}^n$. An IVF  $\textbf{F}: \mathcal{X} \rightarrow I(\mathbb{R})$ is said to be $gH$-Lipschitz continuous on $\mathcal{X}$ if there exists $L~>~0 $ such that
\[ {\lVert \textbf{F}(x) \ominus_{gH} \textbf{F}(y) \rVert }_{I(\mathbb{R})} \le L {\lVert x-y \rVert} ~~\text{for all}~~x,y \in \mathcal{X}. \]
The constant $L$ is called a Lipschitz constant.
\end{dfn}
\begin{dfn}(\emph{$gH$-derivative} \cite{stefanini2009gh}).
Let $\mathcal{X}\subseteq \mathbb{R}$. The $gH$-derivative of an IVF $\textbf{F}:\mathcal{X} \rightarrow I(\mathbb{R})$ at $\bar{x}\in \mathcal{X}$ is defined by
\[
\textbf{F}'(\bar{x})=\displaystyle\lim_{d\rightarrow 0} \frac{\textbf{F}(\bar{x}+d) \ominus_{gH} \textbf{F}(\bar{x})}{d}, ~\text{provided the limit exists.}
\]
\end{dfn}
\begin{rmrk} (See \label{rd1} \cite{chalco2011}).
Let $\mathcal{X}$ be a nonempty subset of $\mathbb{R}$. The $gH$-derivative of an IVF $\textbf{F}:\mathcal{X} \rightarrow I(\mathbb{R})$ at $\bar{x}\in \mathcal{X}$ exists if the derivatives of $\underline{f}$ and $\overline{f}$ at $\bar{x}$ exist and
\[
\textbf{F}'(\bar{x})=\left[\min\left\{\underline{f}'(\bar{x}), \overline{f}(\bar{x})\right\},   \max\left\{\underline{f}'(\bar{x}), \overline{f}(\bar{x})\right\}\right].
\]
However, the converse is not true.
\end{rmrk}
\begin{dfn}\label{pdgh}(\emph{Partial $gH$-derivative } \cite{chalco2013kkt}).
Let $\textbf{F}:\mathcal{X} \rightarrow I(\mathbb{R})$ be an IVF, where $\mathcal{X}$ is a nonempty subset of $\mathbb{R}^n$. Let $\textbf{G}_i$ be defined by
\[
\textbf{G}_i (x_i) = \textbf{F} (\bar{x}_1, \bar{x}_2, \ldots, \bar{x}_{i-1}, x_i, \bar{x}_{i+1}, \ldots, \bar{x}_n),
\]
where $\bar{x} = (\bar{x}_1,\, \bar{x}_2,\, \ldots,\, \bar{x}_n)^T\in\mathcal{X}$. If the $gH$-derivative of $\textbf{G}_i$ exists at $\bar{x}_i$, then the $i$-th  partial $gH$-derivative  of $\textbf{F}$ at $\bar{x}$, denoted $D_i \textbf{F}(\bar{x})$, is defined by
\[
D_i \textbf{F}(\bar{x})=\textbf{G}'_i (\bar{x}_i)~~\text{for all}~~i = 1,\, 2,\, \ldots,\, n.
\]
\end{dfn}
\begin{dfn} (\emph{$gH$-gradient} \cite{chalco2013kkt}).
Let $\mathcal{X}$ be a nonempty subset of $\mathbb{R}^n$. The $gH$-gradient of an IVF $\textbf{F}:\mathcal{X} \rightarrow I(\mathbb{R})$ at a point $\bar{x} \in \mathcal{X}$, denoted $\nabla \textbf{F} (\bar{x})$, is defined by
\[
\nabla \textbf{F} (\bar{x})=\left(   D_1\textbf{F}(\bar{x}),\,
         D_2\textbf{F}(\bar{x}),\, \ldots,\,
         D_n\textbf{F}(\bar{x})
\right)^T.
\]
\end{dfn}
It is to be mentioned that the authors of \cite{chalco2013kkt} used the notations `$\left(\frac{\partial \textbf{F}}{\partial x_i}\right)_g(\bar{x})$' and `$\nabla_g \textbf{F}(\bar{x})$' for $i$-th  partial $gH$-derivative  and $gH$-gradient of $\textbf{F}$ at $\bar{x}$, respectively. However, throughout the article we simply use the notations `$D_i\textbf{F}(\bar{x})$' and `$\nabla \textbf{F}(\bar{x})$' for $i$-th  partial $gH$-derivative  and $gH$-gradient of $\textbf{F}$ at $\bar{x}$, respectively.
\begin{dfn}(\emph{$gH$-Lipschitz gradient} \cite{ghosh2019derivative}). An IVF  $\textbf{F}: \mathcal{X} \rightarrow I(\mathbb{R})$ is said to have $gH$-Lipschitz gradient on $\mathcal{X} \subseteq \mathbb{R}^n$ if there exists $M~>~0 $ such that
\[ {\lVert \nabla\textbf{F}(x) \ominus_{gH} \nabla\textbf{F}(y) \rVert }_{I(\mathbb{R})^n} \le M {\lVert x-y \rVert} ~~\text{for all}~~x,~y \in \mathcal{X}. \]
\end{dfn}
Until now, the concepts of $gH$-continuity, $gH$-derivative, partial $gH$-derivative , $gH$-gradient for IVF have been discussed. In the next section, we illustrate the idea of differentiability for IVFs. This idea differentiability for IVFs is used in the rest of the paper to develop the gradient descent method of IOP.
%
%
%


\section{Differentiability of Interval-valued Functions}\label{sdivf}

\noindent \hl{Behind the concept of differentiability of a function the concept of linearity plays an important role}. Thus, before exploring the concept of differentiability of an IVF, we discuss the concept of a linear IVF.
{\color{red}
\begin{dfn}(\emph{Linear IVF}). \label{dlivf}
Let $\mathcal{X}$ be a linear subspace of $\mathbb{R}^n$. A function $\textbf{F}: \mathcal{X} \rightarrow I(\mathbb{R})$ is said to be linear if
\[\textbf{F}(x)=\bigoplus_{i=1}^n x_i\odot\textbf{F}(e_i)~ \text{for all}~x=(x_1, x_2,\ldots, x_n)^T\in \mathcal{X},\]
where $e_i$ is the $i$-th standard basis vector of $\mathbb{R}^n$, $i = 1, 2, \ldots, n$ and
`$\bigoplus_{i=1}^n$' denotes successive addition of $n$ number of intervals.
\end{dfn}
}
%
%
%
%
%
%
%
%
%
\begin{rmrk}\label{exlivf1}
It is noteworthy that any IVF $\textbf{F}: \mathbb{R}^n \rightarrow I(\mathbb{R})$ of the following form
\[\textbf{F}(x)=\bigoplus_{i=1}^n x_i \odot \textbf{A}_i=\bigoplus_{i=1}^n x_i \odot [\underline{a}_i, \overline{a}_i],
\]
is a linear IVF.
\end{rmrk}
\begin{example}\label{exlinear}
The IVF $\textbf{F}(x): \mathbb{R} \rightarrow I(\mathbb{R})$, which is defined by
\[
\textbf{F}(x)=[-3, 7]\odot x
\]
is a linear IVF, which is depicted in Figure \ref{fig_linearivf} by gray shaded region.
\begin{figure}[H]
\begin{center}
\includegraphics[scale=0.75]{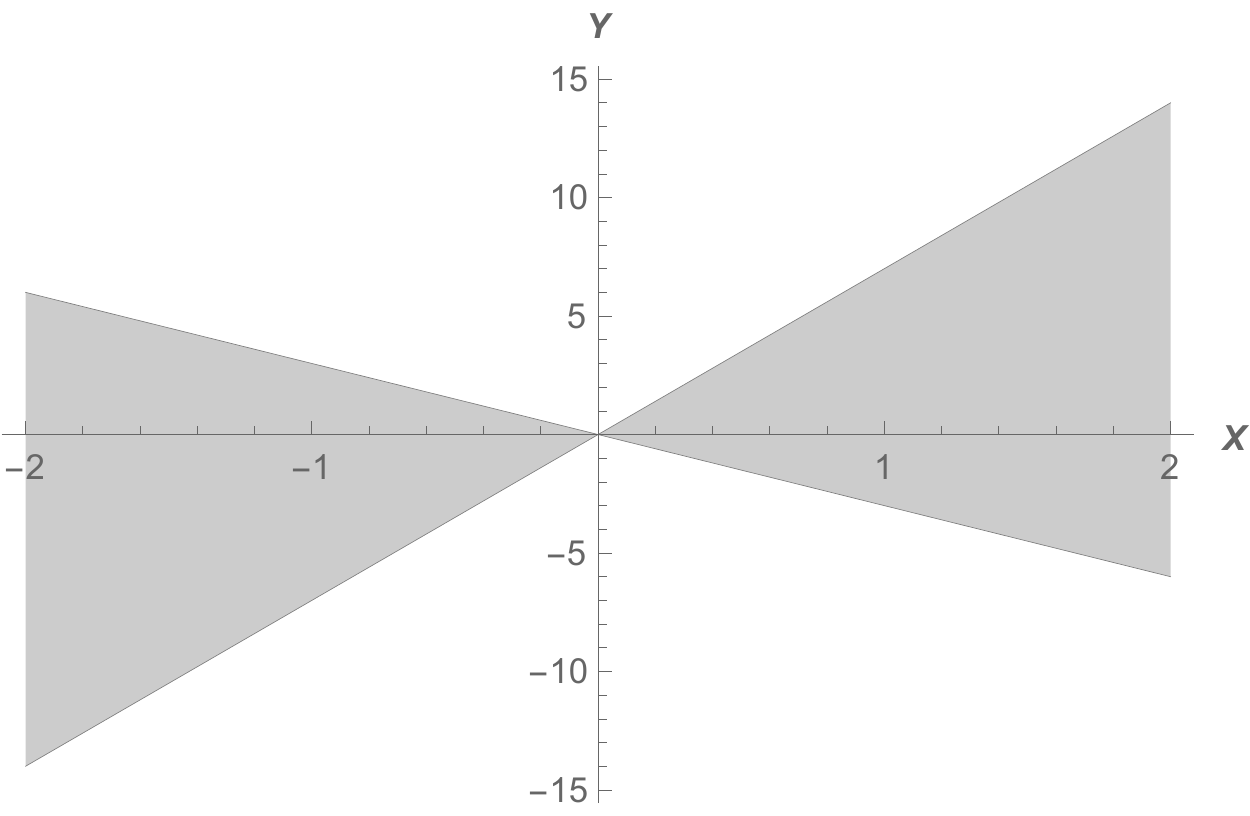}
    \caption{Interval-valued function of Example \ref{exlinear}} \label{fig_linearivf}
\end{center}
\end{figure}

\end{example}
{\color{red}
\begin{rmrk}\label{exlivf2}
A linear IVF $\textbf{F}$ on a linear subspace $\mathcal{X}$ of $\mathbb{R}^n$ satisfies the following conditions: 
\begin{enumerate}[(i)]
\item\label{cl1} $\textbf{F}(\lambda x)=\lambda\odot\textbf{F}(x)~ \text{for all}~x\in \mathcal{X}~\text{and for all}~\lambda \in \mathbb{R}$, and 
\item\label{cl3} for all $x,~y\in \mathcal{X}$, either
\[
\textbf{F}(x+y) = \textbf{F}(x)\oplus\textbf{F}(y) 
\]
 or none of $\textbf{F}(x)\oplus\textbf{F}(y)$ and $\textbf{F}(x+y)$ dominates the other.
\end{enumerate}
For the proof, see \ref{aplivf}.
\end{rmrk}
}
\begin{dfn}\label{dghd} (\emph{$gH$-differentiability}).
Let $\mathcal{X}$ be a nonempty subset of $\mathbb{R}^n$. An IVF $\textbf{F}:\mathcal{X} \rightarrow I(\mathbb{R})$ is said to be $gH$-differentiable at a point $\bar{x} \in \mathcal{X}$ if there exists a linear IVF $\textbf{L}_{\bar{x}}:\mathbb{R}^n\rightarrow I(\mathbb{R})$, an IVF $\textbf{E}(\textbf{F}(\bar{x});d)$ and a $\delta~>~0$ such that
\[
\left(\textbf{F}(\bar{x}+d)\ominus_{gH} \textbf{F}(\bar{x})\right)\ominus_{gH} \textbf{L}_{\bar{x}}(d)=\lVert d \rVert \odot \textbf{E}(\textbf{F}(\bar{x});d)~~\text{for all}~~d~~\text{such that}~~\lVert d \rVert~<~ \delta,
\]
where $\textbf{E}(\textbf{F}(\bar{x});d)\rightarrow \textbf{0}$ as $\lVert d \rVert\rightarrow 0$.\\

\noindent If $\textbf{F}$ is $gH$-differentiable at each point $\bar{x} \in \mathcal{X}$, then $\textbf{F}$ is said to be $gH$-differentiable on $\mathcal{X}$.
\end{dfn}
\begin{rmrk}
It is to note from Definition \ref{dghd} that
\begingroup\allowdisplaybreaks\begin{align*}
&\lim_{\lVert d \rVert \rightarrow 0}\left[\left(\textbf{F}(\bar{x} + d) \ominus_{gH} \textbf{F}(\bar{x})\right) \ominus_{gH} \textbf{L}_{\bar{x}}(d)\right] = \lim_{\lVert d \rVert \rightarrow 0}\textbf{E}(\textbf{F}(\bar{x}); d) \\
\text{or, }&\lim_{\lVert d \rVert \rightarrow 0}\left(\textbf{F}(\bar{x} + d) \ominus_{gH} \textbf{F}(\bar{x})\right) \ominus_{gH} \lim_{\lVert d \rVert \rightarrow 0}\textbf{L}_{\bar{x}}(d) = \textbf{0}\\
\text{or, }&\lim_{\lVert d \rVert \rightarrow 0} \left(\textbf{F}(\bar{x}+d) \ominus_{gH} \textbf{F}(\bar{x}) \right) = \textbf{0}.
\end{align*}\endgroup
Thus, every $gH$-differentiable IVF $\textbf{F}$ is $gH$-continuous.
\end{rmrk}

The following lemma is same as Proposition 7 in \cite{stefanini2019}. However, in \cite{stefanini2019}, Proposition 7 is proved by expressing an IVF $\textbf{F}$ in terms of its midpoint-radius representation, i.e., $\textbf{F}=\left[\tfrac{\underline{f}+\overline{f}}{2}, \tfrac{\overline{f}-\underline{f}}{2}\right]$, but in this article, to prove the following lemma we do not use the midpoint-radius representation of an IVF.

\begin{lem}\label{ld1}
Let $\mathcal{X}$ be a nonempty subset of $\mathbb{R}^n$. If an IVF $\textbf{F}:\mathcal{X}\rightarrow I(\mathbb{R})$ is $gH$-differentiable at $\bar{x}\in \mathcal{X}$, then there exists a nonzero $\lambda$ and $\delta~>~0$ such that
\[
\lim_{\lambda \to 0} \tfrac{1}{\lambda} \odot \left(\textbf{F}(\bar{x} + \lambda h) \ominus_{gH} \textbf{F}(\bar{x})\right)
=\textbf{L}_{\bar{x}}(h)~~\text{for all}~~h \in \mathbb{R}^n ~~\text{with}~~|\lambda|\lVert h \rVert ~<~ \delta,
\]
where $\textbf{L}_{\bar{x}}$ is the linear IVF in Definition \ref{dghd}.
\end{lem}
\begin{proof}
See \ref{apld1}.
\end{proof}
\begin{thm}\label{td1}
Let an IVF $\textbf{F}$ on a nonempty subset $\mathcal{X}$ of $\mathbb{R}^n$ be $gH$-differentiable at $\bar{x}\in \mathcal{X}$. Then, for each $d = (d_1, d_2, \ldots, d_n)^T \in \mathbb{R}^n$, the $gH$-gradient of $\textbf{F}$ at $\bar{x}$ exists and the linear IVF $\textbf{L}_{\bar{x}}$ in Definition \ref{dghd} can be expressed by
\begin{equation}\label{edf3}
\textbf{L}_{\bar{x}}(d) = d^T \odot \nabla\textbf{F}(\bar{x}),
\end{equation}
where $d^T \odot \nabla\textbf{F}(\bar{x})=\bigoplus_{i=1}^n d_i \odot D_i\textbf{F}(\bar{x})$.
\end{thm}
\begin{proof}
See \ref{aptd1}.
\end{proof}
{\color{red} The following example shows that the converse of Theorem \ref{td1} is not true.
\begin{example}\label{exl}
On $\mathbb{R}^2$, consider the following IVF: 
\[
\textbf{F}(x) = x_1\odot\textbf{A}_1\oplus x_2\odot\textbf{A}_2 =x_1\odot[-1, 1]\oplus x_2\odot[0, 2]. 
\]
At $\bar{x} = (0, 1)$,
\[
\textbf{F}(\bar{x})=[0, 2],~D_1\textbf{F}(\bar{x})=\textbf{A}_1=[-1, 1],~D_2\textbf{F}(\bar{x})=\textbf{A}_2=[0, 2].
\]
Therefore, the $gH$-gradient of $\textbf{F}$ at $\bar{x} = (0, 1)$ exists,  $\nabla\textbf{F}(\bar{x})=\left(\textbf{A}_1, \textbf{A}_2\right)^T$ and
\[
d^T\odot\nabla\textbf{F}(\bar{x})=d_1\odot\textbf{A}_1\oplus d_2\odot\textbf{A}_2=\textbf{F}(d)
\text{ for any direction } d\in \mathbb{R}^2.
\]
However, $\textbf{F}$ is not $gH$-differentiable at $\bar{x}$ because at a direction $d=(t, -t)$ with $t>0$ and $\lVert d \rVert<\tfrac{1}{2}$, we obtain
\[
\textbf{F}(\bar{x} + d)=\textbf{F}(t, 1-t)=[-t, 2-t],~\textbf{F}(d)=\textbf{F}(t, -t)=[-3t, t],
\]
and
\begingroup\allowdisplaybreaks\begin{align*}
&\lim_{t\rightarrow 0+}\frac{1}{\lVert d \rVert} \odot\big(\left(\textbf{F}(\bar{x} + d) \ominus_{gH} \textbf{F}(\bar{x})\right)\ominus_{gH}d^T\odot \nabla\textbf{F}(\bar{x})\big)\\
=~&
\lim_{t\rightarrow 0+}\frac{1}{\sqrt{2}t} \odot\big(\left([-t, 2-t] \ominus_{gH} [0, 2]\right)\ominus_{gH} [-3t, t]\big)\\
=~&
\lim_{t\rightarrow 0+}\frac{1}{\sqrt{2}t} \odot\big([-t, -t] \ominus_{gH}[-3t, t]\big)\\
=~&
\lim_{t\rightarrow 0+}\frac{1}{\sqrt{2}t} \odot[-2t, 2t]\\
=~& \sqrt{2}\odot [-1, 1]\\
\neq~& \textbf{0}.
\end{align*}\endgroup
\end{example}
} 
\begin{rmrk}\label{nd1}
By Theorem \ref{td1}, one can notice that the proposed Definition \ref{dghd} of $gH$-differentiability of this article implies the definition of $gH$-differentiability proposed in \cite{stefanini2009}.
 One may think that the definition of $gH$-differentiability of this article is same as that in \cite{ghosh2016newton}. However, it can be noted that the IVF $\textbf{L}_{\bar{x}}$ in \cite{ghosh2016newton} has been considered with the following two properties:
\begin{enumerate}[(a)]
\item\label{cgl1} $\textbf{L}_{\bar{x}}(\lambda x) = \lambda \odot \textbf{L}_{\bar{x}}(x)$ for all $\lambda \in \mathbb{R}$ and $x \in \mathbb{R}^n$ and
\item \label{cgl2} $\textbf{L}_{\bar{x}}(x + y) = \textbf{L}_{\bar{x}}(x) \oplus \textbf{L}_{\bar{x}}(y)$ for all $x,~ y \in \mathbb{R}^n$.
\end{enumerate}
Thus, the IVF $\textbf{L}_{\bar{x}}$ in \cite{ghosh2016newton} is a particular case of the proposed $\textbf{L}_{\bar{x}}$ (see Definition \ref{dlivf}). Hence, the definition of $gH$-differentiability of this article is more general than the definition of \cite{ghosh2016newton}. In the following example, we provide an IVF, which is $gH$-differentiable in the sense of this article but not in the sense of \cite{ghosh2016newton}.\\ 
\end{rmrk}

\begin{example}
Consider the IVF $\textbf{F}:\mathbb{R}\rightarrow I(\mathbb{R})$ which is defined by
\[
\textbf{F}(x)=[-1, 1]\odot x^2,~ x \in \mathbb{R}.
\]
Thus, 
\[
\underline{f}(x)=-x^2~\text{ and }~\overline{f}(x)=x^2.
\]
The $gH$-gradient of $\textbf{F}$ is
\[
\nabla \textbf{F}(x)=[-2, 2]\odot x.
\]
Since both the real-valued functions $\underline{f}$ and $\overline{f}$ are differentiable at $\bar{x}=1$, according to Remark \ref{rd1} the IVF $\textbf{F}$ is $gH$-differentiable at $\bar{x}=1$. Hence, due to Theorem \ref{td1} of this article and Theorem 1 of \cite{ghosh2016newton} there exists an IVF $\textbf{L}_{1}$ such that
\begingroup\allowdisplaybreaks\begin{align*}
\textbf{L}_{1}(h)&=h \odot \nabla \textbf{F}(1)\\
&=h \odot [-2, 2] \odot (1)\\
&=[-2, 2]\odot h, ~~\text{where}~~h\in \mathbb{R}.
\end{align*}\endgroup
By Remark \ref{exlivf1}, $\textbf{L}_{1}$ is a linear IVF. Hence, in the sense of the definition of $gH$-differentiability of this article, $\textbf{F}$ is $gH$-differentiable at $\bar{x}=1$.\\

\noindent However, $\textbf{F}$ is not $gH$-differentiable at $\bar{x}=1$ in the sense of \cite{ghosh2016newton} because {\color{red} there exist some $p$, $q\in \mathbb{R}$} such that
\[
\textbf{L}_{1}(p+q)\neq\textbf{L}_{1}(p)\oplus\textbf{L}_{1}(q).
\]
For instance, consider $p=3$ and $q=-2$. Then,
\[
\textbf{L}_{1}(p+q)=\textbf{L}_{1}(1)=[-2, 2]
\]
and
\[
\textbf{L}_{1}(p)\oplus\textbf{L}_{1}(q)=[-9, 9] \oplus [-6, 6]=[-15, 15]~\neq~[-2, 2].
\]
\end{example}
\begin{rmrk} It may appear as if Theorem $3.1$ of this article is same as Theorem $1$ in \cite{ghosh2016newton} but we have seen that $\textbf{L}_{\bar{x}}$ in \cite{ghosh2016newton} is a particular case of $\textbf{L}_{\bar{x}}$ of this article. Thus, it is clear that Theorem $3.1$ of this article is the generalized version of Theorem $1$ in \cite{ghosh2016newton}.
\end{rmrk}
{\color{red}
\begin{rmrk}
It is noteworthy that although each linear real-valued function is differentiable in its domain, Example \ref{exl} shows that there exists a few linear IVFs that are not $gH$-differentiable.
\end{rmrk}

The following theorem provides a condition for a linear IVF to be $gH$-differentiable.
\begin{thm}\label{td1a}
Let $\mathcal{X}$ be a linear subspace of $\mathbb{R}^n$ and $\textbf{F}$ be an IVF on $\mathcal{X}$. For a given $\bar{x} \in \mathcal{X}$, if for any $d\in \mathcal{N}_\delta(\bar{x})\cap \mathcal{X}$,
\[
    \textbf{F}(\bar{x}+d)=\textbf{F}(\bar{x})\oplus\textbf{F}(d),
\]
where $\mathcal{N}_\delta(\bar{x})$ is a $\delta$-neighborhood of $\bar{x}$, then $\textbf{F}$
is $gH$-differentiable at $\bar{x}\in \mathcal{X}$.  
\end{thm}
\begin{proof}
See \ref{aptd1a}.
\end{proof}
}
\begin{thm}\label{td2}
Let an IVF $\textbf{F}$ on a nonempty open convex subset $\mathcal{X}$ of $\mathbb{R}^n$ be $gH$-differentiable. If the function $\textbf{F}$ is convex on $\mathcal{X}$, then
\[
(y-x)^T \odot \nabla\textbf{F}(x)~\preceq~ \textbf{F}(y)\ominus_{gH}\textbf{F}(x)  ~\text{ for all } x,~y\in \mathcal{X}.
\]
\end{thm}
\begin{proof}
See \ref{aptd2}.
\end{proof}
\begin{thm}\label{td3}
Let an IVF $\textbf{F}$ on a nonempty open convex subset $\mathcal{X}$ of $\mathbb{R}^n$ be $gH$-differentiable on $\mathcal{X}$. If the function $\textbf{F}$ is convex on $\mathcal{X}$, then
\[
\textbf{0}\preceq(x-y)^T \odot \nabla\textbf{F}(x)\ominus_{gH}(x-y)^T \odot \nabla\textbf{F}(y) ~\text{ for all } x,~y\in \mathcal{X}.
\]
\end{thm}
\begin{proof}
See \ref{aptd3}.
\end{proof}
\begin{rmrk}
One may think that for a $gH$-differentiable IVF $\textbf{F}$ on $\mathcal{X}\subseteq\mathbb{R}^n$,
\[
(x-y)^T \odot \nabla\textbf{F}(x)\ominus_{gH}(x-y)^T \odot \nabla\textbf{F}(y)=(x-y)^T \odot \left(\nabla\textbf{F}(x)\ominus_{gH}\nabla\textbf{F}(y)\right)~\text{ for all } x,~y\in \mathcal{X}.
\]
Unfortunately, it is not true in general even if $\textbf{F}$ is convex on $\mathcal{X}$. For instance, consider the following IVF on $\mathbb{R}^2$: 
\[
\textbf{F}(x_1, x_2)=[1, 3]\odot x_1^2\oplus [1, 4]\odot x_2^2=[x_1^2+x_1^2,~ 3x_1^2+4x_2^2]. 
\]
Since $\underline{f}(x_1, x_2) = x_1^2+x_1^2$ and $\overline{f}(x_1, x_2)=3x_1^2+4x_2^2$ are convex on $\mathbb{R}^2$, by Remark \ref{rc1}, $\textbf{F}$ is convex on $\mathbb{R}^2$. \\

\noindent The $gH$-gradient of $\textbf{F}$ is
\[
\nabla\textbf{F}(x)=\left([2, 6]\odot x_1, [2, 8]\odot x_2\right)^T.
\]
Considering $x=(2, 0)$ and $y=(1, 1)$ we have
\begingroup\allowdisplaybreaks\begin{align*}
&(x-y)^T \odot \nabla\textbf{F}(x)\ominus_{gH}(x-y)^T \odot \nabla\textbf{F}(y)\\
=&~(1, -1)\odot\left([4, 12], [0, 0]\right)^T\ominus_{gH}(1, -1)\odot\left([2, 6], [2, 8]\right)^T\\
=&~[4, 12]\oplus[-6, 4]\\
=&~[-2, 16]
\end{align*}\endgroup
and
\begingroup\allowdisplaybreaks\begin{align*}
(x-y)^T\left(\nabla\textbf{F}(x)\ominus_{gH}\nabla\textbf{F}(y)\right)=&(1, -1)\odot\big(\left([4, 12], [0, 0]\right)^T\ominus_{gH}\left([2, 6], [2, 8]\right)^T\big)\\
=&~(1, -1)\odot\left([2, 6], [-8, -2]\right)^T\\
=&~[4, 14]\\
\neq&~[-2, 16].
\end{align*}\endgroup

\end{rmrk}


\section{Interval Optimization Problem and its Efficient Solutions}\label{siopes}

\noindent \hl{This section explores the connection between solutions and the $gH$-derivatives of the following IOP}:
\begin{equation}\label{IOP}
\displaystyle \min_{x \in \mathcal{X}\subseteq\mathbb{R}^n} \textbf{F}(x),
\end{equation}
where $\textbf{F}:\mathcal{X}\rightarrow I(\mathbb{R})$ is a $gH$-differentiable function.\\

\noindent The concept of an efficient solution of the IOP (\ref{IOP}) is defined below.
\begin{dfn}
(\emph{Efficient solution} \cite{wu2008})\label{efficient_sol_def}. A point $\bar{x} \in \mathcal{X}$ is called a global efficient solution of the IOP (\ref{IOP}) if $\textbf{F}(x) ~\nprec~ \textbf{F} (\bar{x})$ for all $x (\neq \bar{x}) \in
\mathcal{X}$.\\

\noindent A point $\bar{x} \in \mathcal{X}$ is called a local efficient solution of IOP (\ref{IOP}) if there exists a $\delta$-neighborhood $N_\delta(\bar{x})$ of $\bar{x}$ such that
\[
\textbf{F}(x) ~\nprec~ \textbf{F} (\bar{x})~~\text{for all}~~x (\neq \bar{x})\in N_\delta(\bar{x})\cap\mathcal{X}.
\]
\end{dfn}
It is to mention that here that Wu \cite{wu2008} named the efficient solution of this article as nondomiated solution. However, throughout this article, we follow Definition \ref{efficient_sol_def} for an efficient solution, and in the rest of the article, by an efficient solution we mean a global efficient solution. \\

{\color{red} Since an IOP is a special case of a fuzzy optimization problem  \cite{ghosh2019analytical}, the following theorem can be considered as a corollary of Theorem $10$ of \cite{stefanini2019}. Further, as differentiability of a function is a special case of  G\^{a}teaux differentiability, the following theorem can also be considered as a corollary of Theorem $4.2$ of \cite{ghosh2019derivative}. In this article, as we are dealing with interval optimization problems and $gH$-differentiability of interval-valued functions, we show the proof of the following theorem to enhance the readability. However, we prove the following theorem with a different approach than \cite{ghosh2019derivative}.} \\ 
\begin{thm}\emph{(Optimality condition).}\label{tes2}
Let $\textbf{F}$ be a $gH$-differentiable IVF on a nonempty subset $\mathcal{X}$ of $\mathbb{R}^n$. If $\bar{x} \in \mathcal{X}$ is an efficient solution of the IOP (\ref{IOP}), then
\[
0\in d^T \odot \nabla\textbf{F}(\bar{x}) ~ \text{ for all} ~ d\in \mathbb{R}^n.
\]
\end{thm}
\begin{proof}
Let $\bar{x} \in \mathcal{X}$ be an efficient solution of the IOP (\ref{IOP}). Therefore, for all $d\in \mathbb{R}^n$ and $\lambda\in\mathbb{R}$ so that $\bar{x}+\lambda d\in\mathcal{X}$, we have
\begingroup\allowdisplaybreaks\begin{align*}
& \textbf{F}(\bar{x}+\lambda d)~\nprec~ \textbf{F}(\bar{x}) \\
\text{or, } & \textbf{F}(\bar{x}+\lambda d)\ominus_{gH}\textbf{F}(\bar{x})~\nprec~ \textbf{0}, \text{ by Lemma \ref{ldr1}} \\
\text{or, } & \lim_{\lambda \rightarrow 0+ } \tfrac{1}{\lambda} \odot \left(\textbf{F}(\bar{x}+\lambda d)\ominus_{gH}\textbf{F}(\bar{x}) \right)~\nprec~ \textbf{0}\\ \text{or, } & \lim_{\lambda \rightarrow 0 } \tfrac{1}{\lambda} \odot \left(\textbf{F}(\bar{x}+\lambda d)\ominus_{gH}\textbf{F}(\bar{x}) \right)~\nprec~ \textbf{0}, \text{ since $\textbf{F}$ is $gH$-differentiable on $\mathcal{X}$}.
\end{align*}\endgroup
Therefore, by Lemma \ref{ld1} and Theorem \ref{td1}, we obtain
\begin{equation}\label{ees3}
d^T \odot \nabla\textbf{F}(\bar{x})~\nprec~ \textbf{0} ~ \text{for all} ~ d\in \mathbb{R}^n.
\end{equation}
Replacing $d$ by $-d$ in (\ref{ees3}) we obtain
\begingroup\allowdisplaybreaks\begin{align*}
& (-d)^T \odot \nabla\textbf{F}(\bar{x})~\nprec~ \textbf{0}\\
\text{or, } & (-1)\odot(d^T \odot \nabla\textbf{F}(\bar{x}))~\nprec~ \textbf{0}.
\end{align*}\endgroup
Thus, by Lemma \ref{ldr2},
\begin{equation}\label{ees4}
\textbf{0}~\nprec~ d^T \odot \nabla\textbf{F}(\bar{x}).
\end{equation}
By (\ref{ees3}) and (\ref{ees4}), we get
$
0\in d^T \odot \nabla\textbf{F}(\bar{x}) ~ \text{for all} ~ d\in \mathbb{R}^n.
$
\end{proof}
\begin{rmrk}
It is noteworthy that the converse of Theorem \ref{tes2} is not always true even if $\textbf{F}$ is a convex IVF. For instance, let us consider the following IOP:
\begin{equation}\label{ex1iop}
\min_{x\in \mathbb{R}} \textbf{F}(x),
\end{equation}
where
\[
\textbf{F}(x)=\begin{cases}
    [0, 3] \ominus_{gH} [-1, 0]\odot x^2, & \text{if } -1 ~\leq~ x ~\leq~ 1 \\
    [0, 2] \oplus [1, 1]\odot x^2,              & \text{otherwise}.
    \end{cases}
\]
\begin{figure}[H]
\begin{center}
\includegraphics[scale=0.65]{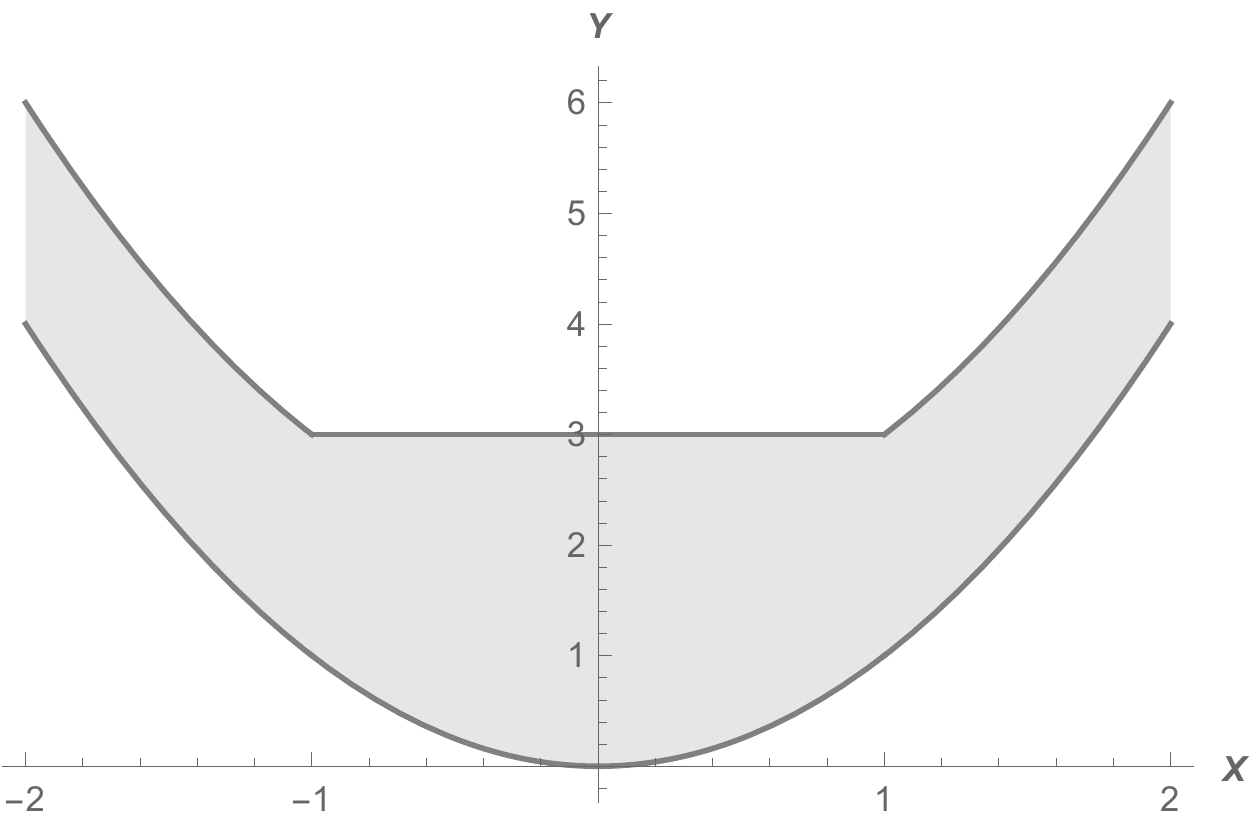}
    \caption{Interval-valued function of the IOP (\ref{ex1iop})} \label{fgm1}
\end{center}
\end{figure}
The graph of the IVF $\textbf{F}$ is depicted in Figure \ref{fgm1} by the shaded region. From Figure \ref{fgm1} it is clear that the IVF $\textbf{F}$ is convex since its lower and upper functions are convex. \\

\noindent The $gH$-gradient of $\textbf{F}$ is
\[
\nabla \textbf{F}(x)=\begin{cases}
    [0, 2]\odot x, & \text{if } -1 ~\leq~ x ~\leq~ 1 \\
    [2, 2]\odot x,              & \text{otherwise}.
    \end{cases}
\]
Thus, at $x=-1$,
\[
\nabla \textbf{F}(-1)=[-2, 0]~\Longrightarrow~ 0\in d\odot\nabla \textbf{F}(-1)~\text{for all}~d\in\mathbb{R}.
\]
But it is notable that for $0~<~h~<~2$,
\[
\textbf{F}(-1+h)=[(-1+h)^2, 3]\prec[1, 3]=\textbf{F}(-1).
\]
Therefore, although $0\in d\odot\nabla \textbf{F}(-1)~\text{for all}~d\in\mathbb{R}$, $-1$ is not an efficient solution of the IOP (\ref{ex1iop}).\\
\end{rmrk}
\begin{cor}\label{cred2}
Let $\textbf{F}$ be a $gH$-differentiable IVF on a nonempty subset $\mathcal{X}$ of $\mathbb{R}^n$. If $\bar{x} \in \mathcal{X}$ is an efficient solution of the IOP (\ref{IOP}), then
\[
0\in D_i \textbf{F}(\bar{x}) ~~ \text{for each} ~~ i\in \{1, 2, \ldots, n\}.
\]
\end{cor}
\begin{proof}
Let $\bar{x} \in \mathcal{X}$ be an efficient solution of the IOP (\ref{IOP}). According to Theorem \ref{tes2}, for all $d\in \mathbb{R}^n$,  we have
\[
0 \in d^T \odot \nabla\textbf{F}(\bar{x})=\bigoplus_{i=1}^n d_i \odot D_i\textbf{F}(\bar{x}).
\]
For each $i\in \{1, 2, \ldots, n\}$, by considering $d=e_i$, we obtain
\[
0 \in D_i\textbf{F}(\bar{x}).
\]
\end{proof}
%
%
%


\section{$gH$-gradient Efficient Methods for Interval Optimization Problem}\label{sgdmiop}

\noindent \hl{This section develops $gH$-gradient efficient techniques to obtain the efficient solutions of the IOP} (\ref{IOP}). In the conventional gradient descent technique, to find a minimizer,  we move sequentially along descent directions. Likewise, for IOP, to find an efficient solution we attempt to move  sequentially along efficient-directions, defined below.
\begin{dfn}\label{ded}
(\emph{Efficient-direction}). Let $\mathcal{X}\subseteq\mathbb{R}^n$. A direction $d\in \mathbb{R}^n$ is said to be an efficient-direction of an IVF $\textbf{F}:\mathcal{X}\rightarrow I(\mathbb{R})$ at $\bar{x} \in \mathcal{X}$ if there exists a $\delta~>~0$ such that
\begin{enumerate}[(i)]
\item\label{ced1} $\textbf{F} (\bar{x}) ~\npreceq~ \textbf{F}(\bar{x}+\lambda d) ~\text{for all}~\lambda \in (0,\delta),$

\item\label{ced2} there also exists a point $x'=\bar{x}+\alpha d$ with $\alpha \in (0,\delta)$ and a positive real number $\delta'~\leq~ \alpha$ such that
     \[
      \textbf{F} (x'+\lambda d) ~\nprec~ \textbf{F}(x')~\text{for all}~\lambda \in (-\delta',\delta').
     \]
     The point $x'$ is known as an efficient point of $\textbf{F}$ in the direction $d$.
\end{enumerate}
\end{dfn}
In Figure \ref{fed}, the points $\bar{x}$ and $x'$, the direction $d$, and the nonnegative real numbers $\delta$ and $\delta'$ of Definition \ref{ded} are illustrated on $\mathbb{R}^2$ plane ($n = 2$).
\begin{figure}[!h]
\begin{center}
\includegraphics[scale=.7]{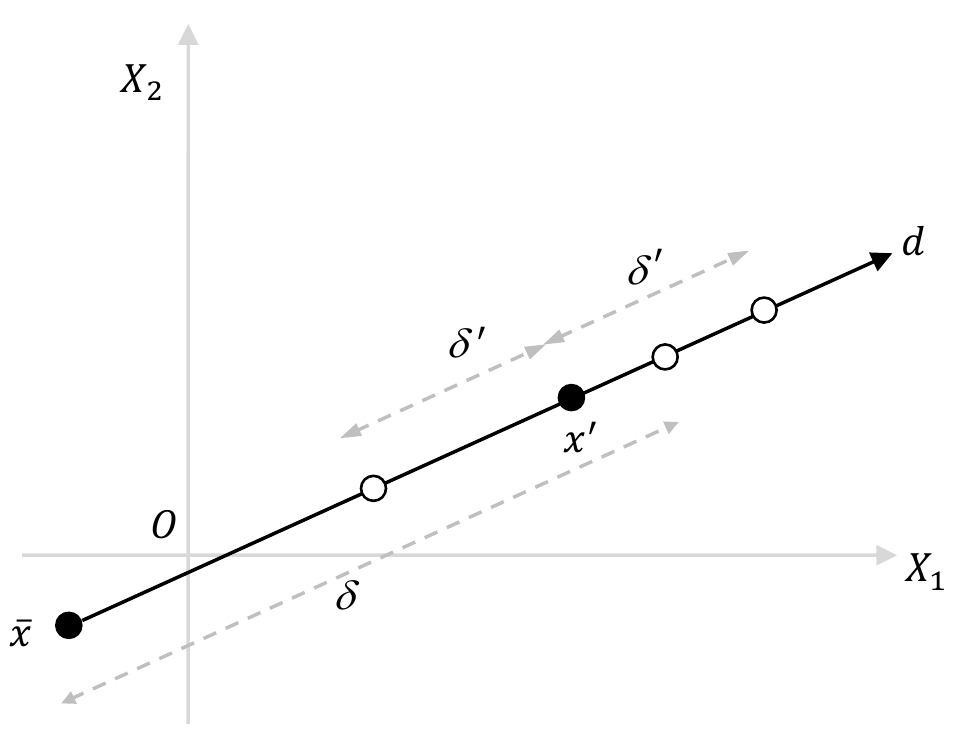}
\caption{Locations of $\bar{x}$ and $x'$ in Definition \ref{ded}}\label{fed}
\end{center}
\end{figure}
\begin{rmrk}
One may think that only the condition (\ref{ced1}) of Definition \ref{ded} is sufficient to define an efficient-direction. However, it is not true in general. Because, for $\textbf{A}$, $\textbf{B}$ and $\textbf{C}\in I(\mathbb{R})$,
\[
\textbf{A}~\npreceq~ \textbf{B}~~\text{and}~~ \textbf{B}~\npreceq~ \textbf{C}~\not\Longrightarrow~ \textbf{A}~\npreceq~ \textbf{C}~~\text{in general}.
\]
For instance, consider
\[
\textbf{A}=[4, 6],~~\textbf{B}=[2, 10]~~\text{and}~~ \textbf{C}=[5, 7].
\]
We, then, see that
\[
\textbf{A}~\npreceq~ \textbf{B}~~\text{and}~~ \textbf{B}~\npreceq~ \textbf{C}~~\text{but}~~ \textbf{A}~\prec~ \textbf{C}.
\]
That is why the condition (\ref{ced2}) of Definition \ref{ded} is necessary to define an efficient direction.
\end{rmrk}
\begin{rmrk} For the degenerate case of the IVF $\textbf{F}$, i.e., for $\underline{f}(x)=\overline{f}(x)=f(x)$ for all $x\in\mathcal{X}$, Definition \ref{ded} reduces to the following. A direction $d\in \mathbb{R}^n$ is said to be an efficient-direction of $f$ if there exists a $\delta~>~0$ such that
\[
f(\bar{x})~>~f(\bar{x}+\lambda d)~~\text{for all}~~\lambda \in (0,\delta).
\]
Thus, an efficient direction for a degenerate IVF is a descent direction.
\end{rmrk}
\begin{thm}\label{tndd1}
Let $\textbf{F}$ be a $gH$-differentiable IVF on a nonempty subset $\mathcal{X}$ of $\mathbb{R}^n$. Then, every direction $d\in \mathbb{R}^n$ that satisfies
\begin{equation}\label{endd1a}
\textbf{0}~\npreceq~ d^T \odot \nabla\textbf{F}(\bar{x}),
\end{equation}
is an efficient-direction of $\textbf{F}$ at $\bar{x}\in \mathcal{X}$, where the corresponding efficient point $x'$ is provided by
\[
x'=\bar{x}+\alpha' d~~\text{with}~~\alpha'=\argeff_{\alpha\in \mathbb{R}_+}~\textbf{F}(\bar{x}+\alpha d),
\]
where by `$\argeff$' of $\textbf{F}(\bar{x}+\alpha d)$, we mean a point $\alpha'$ such that
\[
\textbf{F}(\bar{x}+\alpha d)~\nprec~ \textbf{F}(\bar{x}+\alpha' d)~~\text{for all}~~\alpha\in\mathbb{R}.
\]
\end{thm}
\begin{proof}
Let $d\in \mathbb{R}^n$ be a direction that satisfies the relation (\ref{endd1a}). Since $\textbf{F}$ is $gH$-differentiable at $\bar{x}\in \mathcal{X}$, by Lemma \ref{ld1} and Theorem \ref{td1}, we have
\[
\lim_{\lambda\rightarrow 0}\frac{1}{\lambda} \odot \left(\textbf{F}(\bar{x}+\lambda d)\ominus_{gH} \textbf{F}(\bar{x})\right)= d^T \odot \nabla\textbf{F}(\bar{x}),
\]
which implies
\[
\lim_{\lambda\rightarrow 0+}\frac{1}{\lambda} \odot \left(\textbf{F}(\bar{x}+\lambda d)\ominus_{gH} \textbf{F}(\bar{x})\right)= d^T \odot \nabla\textbf{F}(\bar{x}).
\]
Due to the relation (\ref{endd1a}), the last equation yields
\begin{equation}\label{endd3}
\textbf{0} ~\npreceq~ \lim_{\lambda\rightarrow 0+}\frac{1}{\lambda} \odot \left(\textbf{F}(\bar{x}+\lambda d)\ominus_{gH} \textbf{F}(\bar{x})\right).
\end{equation}
Therefore,
\begingroup\allowdisplaybreaks\begin{align*}
& [0, 0] ~\npreceq~ \lim_{\lambda\rightarrow 0+}\frac{1}{\lambda} \odot\Big[\min \left\{\underline{f}(\bar{x}+\lambda d)-\underline{f}(\bar{x}), \overline{f}(\bar{x}+\lambda d)-\overline{f}(\bar{x})\right\},\\
&~~~~~~~~~~~~~~~~~~~~~~~~~\max \left\{\underline{f}(\bar{x}+\lambda d)-\underline{f}(\bar{x}), \overline{f}(\bar{x}+\lambda d)-\overline{f}(\bar{x})\right\}\Big]\\
\text{or, } & [0, 0] ~\npreceq~ \bigg[\min \left\{\lim_{\lambda\rightarrow 0+}\frac{1}{\lambda} \left(\underline{f}(\bar{x}+\lambda d)-\underline{f}(\bar{x})\right), \lim_{\lambda\rightarrow 0+}\frac{1}{\lambda}\left(\overline{f}(\bar{x}+\lambda d)-\overline{f}(\bar{x})\right)\right\},\\
&~~~~~~~~~~~~~\max \left\{\lim_{\lambda\rightarrow 0+}\frac{1}{\lambda}\left(\underline{f}(\bar{x}+\lambda d)-\underline{f}(\bar{x})\right), \lim_{\lambda\rightarrow 0+}\frac{1}{\lambda}\left(\overline{f}(\bar{x}+\lambda d)-\overline{f}(\bar{x})\right)\right\}\bigg],
\end{align*}\endgroup
which implies
\begin{equation}\label{endd4}
\min \left\{\lim_{\lambda\rightarrow 0+}\frac{1}{\lambda} \left(\underline{f}(\bar{x}+\lambda d)-\underline{f}(\bar{x})\right), \lim_{\lambda\rightarrow 0+}\frac{1}{\lambda}\left(\overline{f}(\bar{x}+\lambda d)-\overline{f}(\bar{x})\right)\right\}~<~0.
\end{equation}
Thus, we have following two cases.
\begin{enumerate}[$\bullet$ \textbf{Case} 1.]
\item\label{cs1tndd1}
Let
\begingroup\allowdisplaybreaks\begin{align*}
& \min \left\{\lim_{\lambda\rightarrow 0+}\frac{1}{\lambda} \left(\underline{f}(\bar{x}+\lambda d)-\underline{f}(\bar{x})\right), \lim_{\lambda\rightarrow 0+}\frac{1}{\lambda}\left(\overline{f}(\bar{x}+\lambda d)-\overline{f}(\bar{x})\right)\right\}\\
&~~~~~~~~~~~~~~~~~=\lim_{\lambda\rightarrow 0+}\frac{1}{\lambda} \left(\underline{f}(\bar{x}+\lambda d)-\underline{f}(\bar{x})\right).
\end{align*}\endgroup
Then, by the equation (\ref{endd4}), we obtain
$
\lim_{\lambda\rightarrow 0+}\frac{1}{\lambda} \left(\underline{f}(\bar{x}+\lambda d)-\underline{f}(\bar{x})\right)~<~0.
$
Therefore, there exists a $\delta_1~>~0$ such that for all $\lambda\in (0, \delta_1)$,
\begingroup\allowdisplaybreaks\begin{align*}
& \frac{1}{\lambda} \left(\underline{f}(\bar{x}+\lambda d)-\underline{f}(\bar{x})\right)~<~0\\
\text{or, } & \underline{f}(\bar{x}+\lambda d)-\underline{f}(\bar{x})~<~0\\
\text{or, } & \underline{f}(\bar{x}+\lambda d)~<~\underline{f}(\bar{x})\\
\text{or, } & \left[\underline{f}(\bar{x}), \overline{f}(\bar{x})\right]~\npreceq~ \left[\underline{f}(\bar{x}+\lambda d), \overline{f}(\bar{x}+\lambda d)\right].
\end{align*}\endgroup

\item\label{cs2tndd1} Let
\begingroup\allowdisplaybreaks\begin{align*}
& \min \left\{\lim_{\lambda\rightarrow 0+}\frac{1}{\lambda} \left(\underline{f}(\bar{x}+\lambda d)-\underline{f}(\bar{x})\right), \lim_{\lambda\rightarrow 0+}\frac{1}{\lambda}\left(\overline{f}(\bar{x}+\lambda d)-\overline{f}(\bar{x})\right)\right\}\\
&~~~~~~~~~~~~~~~~~=\lim_{\lambda\rightarrow 0+}\frac{1}{\lambda}\left(\overline{f}(\bar{x}+\lambda d)-\overline{f}(\bar{x})\right).
\end{align*}\endgroup
Then, by the equation (\ref{endd4}), we get
\[
\lim_{\lambda\rightarrow 0+}\frac{1}{\lambda} \left(\overline{f}(\bar{x}+\lambda d)-\overline{f}(\bar{x})\right)~<~0.
\]
Thus, there exists a $\delta_2~>~0$ such that for $\lambda\in (0, \delta_2)$,
\begingroup\allowdisplaybreaks\begin{align*}
& \frac{1}{\lambda} \left(\overline{f}(\bar{x}+\lambda d)-\overline{f}(\bar{x})\right)~<~0\\
\text{or, } & \overline{f}(\bar{x}+\lambda d)-\overline{f}(\bar{x})~<~0\\
\text{or, } & \overline{f}(\bar{x}+\lambda d)~<~\overline{f}(\bar{x})\\
\text{or, } & \left[\underline{f}(\bar{x}), \overline{f}(\bar{x})\right]~\npreceq~ \left[\underline{f}(\bar{x}+\lambda d), \overline{f}(\bar{x}+\lambda d)\right]
\end{align*}\endgroup
\end{enumerate}

\noindent Choosing $\delta=\min \{\delta_1, \delta_2\}$, from Case \ref{cs1tndd1} and Case \ref{cs2tndd1}, we see that for all $\lambda\in (0, \delta)$,
\begingroup\allowdisplaybreaks\begin{align*}
&\left[\underline{f}(\bar{x}), \overline{f}(\bar{x})\right]~\npreceq~ \left[\underline{f}(\bar{x}+\lambda d), \overline{f}(\bar{x}+\lambda d)\right]\\
\text{or, } & \textbf{F}(\bar{x})~\npreceq~ \textbf{F}(\bar{x}+\lambda d).
\end{align*}\endgroup
Hence, $d$ satisfies the condition (\ref{ced1}) of Definition \ref{ded} at $\bar{x}$.\\ \\ 
Further, let us choose an $\alpha'$ such that
\[
\alpha'=\argeff_{\alpha\in \mathbb{R}_+}~\textbf{F}(\bar{x}+\alpha d).
\]
Therefore, there exists a $\bar{\delta}~>~0$ such that for all $\alpha \in (\alpha'-\bar{\delta}, \alpha'+\bar{\delta})$,
\[
\textbf{F}(\bar{x}+\alpha d) ~\nprec~ \textbf{F}(\bar{x}+\alpha' d).
\]
Considering $\lambda=\alpha-\alpha'$, we have
\begingroup\allowdisplaybreaks\begin{align*}
& \textbf{F}(\bar{x}+\alpha' d+\lambda d) ~\nprec~ \textbf{F}(\bar{x}+\alpha' d)\\
\text{or, } & \textbf{F}(x'+\lambda d) ~\nprec~ \textbf{F}(x'),
\end{align*}\endgroup
where $x'= \bar{x}+\alpha' d$. Choosing $\delta'=\min\{\alpha', \bar{\delta}\}$, we have
\[
\textbf{F}(x'+\lambda d) ~\nprec~ \textbf{F}(x')~~\text{for all}~~\lambda\in (-\delta', \delta').
\]
Therefore, $d$ satisfies the condition (\ref{ced2}) of Definition \ref{ded} at $\bar{x}$. Hence, $d$ is an efficient-direction at $\bar{x}$.
\end{proof}%
\begin{rmrk}\label{ngd1}
A question may arise here: in the definition of efficient-direction (Definition \ref{ded}), whether or not the condition (\ref{ced1}) can be replaced by
\begin{equation}\label{endd5}
\textbf{F}(\bar{x}+\lambda d)~\prec~ \textbf{F} (\bar{x}) ~\text{for all}~\lambda \in (0,\delta)\text{?}
\end{equation}
To answer, we note that if we choose the relation (\ref{endd5}) in place of the condition (\ref{ced1}) of Definition \ref{ded}, then in the same way of proving Theorem \ref{tndd1} it can be proved that any $d\in\mathbb{R}^n$ that satisfies
\begin{equation}\label{endd6}
d^T\odot\nabla\textbf{F}(\bar{x})~\prec~ \textbf{0}
\end{equation}
also holds the relation (\ref{endd5}) and vice versa for a $gH$-differentiable IVF $F$. However, the relation (\ref{endd1a}) is more general than the relation (\ref{endd6}). Because there are some directions, along which there exists an efficient solution of an IOP, satisfy the relation (\ref{endd1a}) but do not satisfy the relation (\ref{endd6}).\\

\noindent For instance, consider the IOP:
\begin{equation}\label{ex1aiop}
\min_{x\in\mathcal{X}}\textbf{F}(x)=[-1, 2]\odot x_1^2 \oplus [3, 4]\odot x_2^2,
\end{equation}
where $\mathcal{X}=[-10, 10]\times [-10, 10]\subseteq\mathbb{R}^2$.\\

\noindent In what follows, we show that $(\bar{x}_1, 0)^T\in \mathcal{X}$ is an efficient solution of the IOP (\ref{ex1aiop}) for any $\bar{x}_1\in[-10, 10]$. On contrary, let there exist two nonzero numbers $h_1$ and $h_2$ with $\left(\bar{x}_1+h_1, h_2\right)^T\in \mathcal{X}$ such that
\begingroup\allowdisplaybreaks\begin{align*}
& \textbf{F}(\bar{x}_1+h_1, h_2) ~\prec~ \textbf{F}(\bar{x}_1, 0)\\
\text{or, } & [-1, 2]\odot (\bar{x}_1+h_1)^2 \oplus [3, 4]h_2^2 ~\prec~ [-1, 2]\odot \bar{x}_1^2\\
\text{or, } & [-\bar{x}_1^2-2\bar{x}_1h_1-h_1^2+3h_2^2, 2x_1^2+4\bar{x}_1h+2h_1^2+4h_2^2] ~\prec~ [-\bar{x}_1^2, 2\bar{x}_1^2].
\end{align*}\endgroup
This implies
\[
-2x_1h_1-h_1^2+3h_2^2~\leq~ 0~~\text{and}~~ 4x_1h_1+2h_1^2+4h_2^2~\leq~ 0.
\]
Hence, $h_2^2~\leq~ 0$, which is not possible as $h_2~\neq~ 0$. So, there does not exist any $x\in\mathcal{X}$, which strictly dominates any $(\bar{x}_1, 0)^T\in \mathcal{X}$. Thus, $(\bar{x}_1, 0)^T$
is an efficient solution of the IOP (\ref{ex1aiop}).\\

\noindent Now we choose a point $\hat{x}=(3, 2)^T$ and two directions $d'=(1, -2)^T$ and $d''=(5, -2)^T$. We also choose an $\alpha=1$. As
\[
\textbf{F}(3, 0)=[-1, 2]~\prec~ [5, 10]=\textbf{F}(3, 2)=\textbf{F}(\hat{x}),
\]
$\hat{x}$ is not an efficient solution of the IOP (\ref{ex1aiop}). But the points $x'=\hat{x}+\alpha d'=(4, 0)^T$ and $x''=\hat{x}+\alpha d''=(8, 0)^T$ both are efficient solutions of the IOP (\ref{ex1aiop}). Therefore, both $d'$ and $d''$ are efficient-directions of $\textbf{F}$ at $\hat{x}$. Further, as the $gH$-gradient of $\textbf{F}$ is
\[
\nabla \textbf{F}(x)=(D_1\textbf{F}(x), D_2\textbf{F}(x))^T=\left ([-2, 4]\odot x_1,~[6, 8]\odot x_2\right )^T,
\]
we have $\nabla\textbf{F}(\hat{x})=([-6, 12], [12, 16])^T$. Therefore,
\begingroup\allowdisplaybreaks\begin{align*}
& d'\odot\nabla\textbf{F}(\hat{x})=[-6, 12]\oplus[-32, -24]=[-38, -12]\\
~\Longrightarrow~ &d'\odot\nabla\textbf{F}(\hat{x})~\prec~ \textbf{0}~~\text{and also,}~~\textbf{0}~\npreceq~ d'\odot\nabla\textbf{F}(\hat{x}).
\end{align*}\endgroup
Again,
\begingroup\allowdisplaybreaks\begin{align*}
& d''\odot\nabla\textbf{F}(\hat{x})=[-30, 60]\oplus[-32, -24]=[-38,36]\\
~\Longrightarrow~ &\textbf{0}~\npreceq~ d''\odot\nabla\textbf{F}(\hat{x})
\end{align*}\endgroup
Hence, it is clear that although both $d'$ and $d''$ are efficient-directions of the IVF $\textbf{F}$ at the point $\hat{x}$ and satisfy the relation (\ref{endd1a}) but only $d'$ satisfies the relation (\ref{endd6}). So, the condition (\ref{ced1}) of Definition \ref{ded} is more general than the condition (\ref{endd6}).
\end{rmrk}
%
%
%

\subsection{General $gH$-gradient Efficient-Direction Method for Interval Optimization Problems}

\noindent To produce the efficient solutions of the IOP (\ref{IOP}) we provide Algorithm \ref{algoes1}. As Algorithm \ref{algoes1}
\begin{enumerate}[(i)]
\item uses $gH$-gradient at every iterative step  and
\item endeavors to find an efficient solution by sequentially moving along efficient-directions,
\end{enumerate}
we name the method as general $gH$-gradient efficient-direction method. The term `general' is due to the reason that we do not choose a special or a particular efficient-direction $d_k$; any general $d_k$ that satisfies $\textbf{0}~\npreceq~ d_k^T\odot\nabla\textbf{F}(x_k)$ will lead to reaching at an efficient point.
\begin{algorithm}[H]
    \caption{General $gH$-gradient efficient-direction method for IOP}\label{algoes1}
    \begin{algorithmic}[1]
      \REQUIRE Given the initial point $x_{0}$ and the IVF $\textbf{F}:\mathcal{X}(\subseteq\mathbb{R}^n)\rightarrow I(\mathbb{R})$.
      \STATE Set $k = 0$.
      \STATE If $0\in d^T \odot \nabla\textbf{F}(x_{k})$ for all $d\in \mathbb{R}^n$, then Return $x_{k}$ as an efficient-solution and Stop.              Otherwise go to Step 3.
      \STATE Find a $d_{k}$ such that $\textbf{0}~\npreceq~ d_{k}^T \odot \nabla\textbf{F}(x_{k})$ and an $\alpha_{k}$ such that
            \[
            \alpha_{k}=\argeff_{\alpha\in \mathbb{R}_+}~\textbf{F}(x_{k}+\alpha d_{k})
            \]
      \STATE Calculate
             \[
             x_{k+1} = x_{k} + \alpha_{k}d_{k}.
             \]
      \STATE Set $k \leftarrow k + 1$ and go to Step 2.
    \end{algorithmic}
  \end{algorithm}

\noindent In the next, we give the convergence analysis of the Algorithm \ref{algoes1}. Towards the convergent analysis, we need the following notions of the algorithmic map and the closed map regarding IVFs.

{\color{red} \begin{dfn}(\emph{Algorithmic map} \cite{bazara1993}). Let $\mathcal{X}$ be a nonempty subset of $\mathbb{R}^n$. An algorithmic map  $\mathcal{A}$ of an algorithm is a point-to-set map on its domain $\mathcal{X}$ which describes the iterating process of the algorithm such that if the sequence $\{x_k\}$ is generated by the algorithm then $x_{k+1} \in \mathcal{A}(x_k)$.
\end{dfn}
\begin{rmrk}
The map $\mathcal{A}:\mathcal{X} \rightarrow \mathcal{X}$ that generates the sequence $\{x_k\}$ in Algorithm \ref{algoes1} with $\textbf{F}(x_k)~\npreceq~ \textbf{F}(x_{k+1})$ is an algorithmic map.
\end{rmrk}
}
\begin{dfn} (\emph{Closed map} \cite{bazara1993}). A point-to-set map $\mathcal{A}$ from a nonempty subset $\mathcal{X}$ of $\mathbb{R}^n$ to a subset $\mathcal{Y}$ of $\mathbb{R}^m$ is said to be closed at $x \in \mathcal{X}$ if for any sequences $\{x_k\}$ and $\{y_k\}$ such that
\[
x_k \to x ~~\text{and}~~y_k \to y,~~\text{where}~~x_k\in \mathcal{X}~~\text{and}~~y_k \in \mathcal{A}(x_k),
\]
we have $y\in \mathcal{A}(x)$.
\end{dfn}
\begin{thm}\label{tcm}
Let $\textbf{F}: \mathbb{R}^n \rightarrow I(\mathbb{R})$ be an IVF on a nonempty open subset $\mathcal{X}$ of $\mathbb{R}^n$ and $\textbf{D}$ be an element of $I(\mathbb{R}_+)$. Define a point-to-set map $\mathcal{L}:\mathbb{R}^n \times \mathbb{R}^n\rightarrow \mathbb{R}^n$ by $\mathcal{L}(x, d)=\{y:y=x+\bar{\lambda}d\}$ for some $\bar{\lambda}\in \textbf{D}$. Let $\textbf{F}(x+\lambda d)~\nprec~ \textbf{F}(y)$ for each $\lambda \in \textbf{D}$ and $d~\neq~ 0$. If $\textbf{F}$ is $gH$-continuous at $x$ and $d~\neq~ 0$, then $\mathcal{L}$ is closed at $(x, d)$.
\end{thm}
\begin{proof}
Let $\{(x_k, d_k)\}$ be a sequence such that $(x_k, d_k) \to (x, d)$. Let $\{y_k\}$, $y_k\in \mathcal{L}(x_k, d_k)$, be a sequence such that $y_k\to y$. To prove the theorem, we have to show that $y\in \mathcal{L}(x, d)$.\\

 It is to note that there exists $\lambda_k\in\textbf{D}$ such that $y_k=x_k+\lambda_k d_k$ for $k=1,\ 2,\ \ldots, \ n$. Since $d~\neq~ 0$, for large enough $k$ we have $d_k~\neq~ 0$. Then, $\lambda_k= \frac{\lVert y_k-x_k\rVert}{\lVert d_k \rVert}$.\\

 Taking the limit as $k\rightarrow \infty$ we have $\lambda_k \to \frac{\lVert y-x\rVert}{\lVert d \rVert}=\bar{\lambda}$, say. Hence, $y=x+\bar{\lambda}d$. Furthermore, since $\lambda_k \in \textbf{D} $ for each $k$ and $\textbf{D}$ is closed, $\bar{\lambda}\in \textbf{D}$.\\

 Therefore, as $\lambda \in \textbf{D}$, $\textbf{F}(x_k+\lambda d_k)~\nprec~ \textbf{F}(y_k)~~\text{for all }~k$, and $\textbf{F}$ is $gH$-continuous, by Lemma \ref{lsc}, evidently, we have $\textbf{F}(x+\lambda d)~\nprec~ \textbf{F}(y)$. Hence, $y\in \mathcal{L}(x, d)$ and so, $\mathcal{L}$ is closed.\\
\end{proof}
%
%
%
\begin{rmrk}\label{ngem1} As according to Theorem \ref{tcm}, the map $\mathcal{L}$ is closed, the composite map $\mathcal{A}=\mathcal{L}\circ\mathcal{D}$ will be closed if the direction generating map $\mathcal{D}:\mathbb{R}^n \rightarrow \mathbb{R}^{n}\times \mathbb{R}^{n}$, defined by $\mathcal{D}(x)=(x, d)$, is also closed.
\end{rmrk}
%
%
%
%
%
%
%
%
%
\begin{thm}\emph{(Convergence of general $gH$-gradient efficient method for IOP).}
Let $\textbf{F}$ be a $gH$-differentiable IVF on a nonempty open subset $\mathcal{X}$ of $\mathbb{R}^n$. Suppose $\Omega \subseteq \mathcal{X} $ be the set of all efficient points of $\textbf{F}$ and $\mathcal{A}:\mathcal{X} \rightarrow \mathcal{X}$ be an algorithmic map of Algorithm \ref{algoes1}. Suppose that the algorithm map $\mathcal{A}$ produces the sequence $\{x_k\}$, which converges at $\bar{x} \in \mathcal{X}$. Also, assume that there exists $M~>~0$ such that $\lVert d_k \rVert ~<~M$  for all $k$. If $\bar{d}$ is an accumulation point of ${d_k}$, then we have
\begin{equation}\label{endd2a}
0\in {\bar{d}}^T \odot \nabla\textbf{F}(\bar{x}).
\end{equation}
\end{thm}
\begin{proof}
Since $\{d_k\}$ is bounded, there exists an index set $K_1$ such that $\lim\limits_{k\in K_1}d_k=\bar{d}$ and $\lim\limits_{k\in K_1}x_k=\bar{x}$. We have the following two cases.
\begin{enumerate}[$\bullet$ \textbf{Case} 1.]
\item \emph{If $\bar{d}=0$}. Then, (\ref{endd2a}) is trivial.

\item \emph{If $\bar{d}~\neq~ 0$}. If possible, let us assume that the conclusion (\ref{endd2a}) is not true. Hence,
    \[
    0\notin {\bar{d}}^T \odot \nabla\textbf{F}(\bar{x}) \Longrightarrow \bar{x}\notin \Omega.
    \]

    Since, for $k\in K_1, x_k\to \bar{x}$, also $x_{k+1}\in \mathcal{A}(x_k)$ and $x_{k+1}\to \bar{x}$, therefore, $\bar{x}\in \mathcal{A}(\bar{x})$. As $\mathcal{A}$ is closed at $\bar{x}$ due to Remark \ref{ngem1}. Thus, $\textbf{F}(\bar{x})~\npreceq~ \textbf{F}(\bar{x})$, which is a contradiction. Hence,
    $
    0\in {\bar{d}}^T \odot \nabla\textbf{F}(\bar{x}).
    $
\end{enumerate}
\end{proof}
%
%
%

\subsection{$\mathcal{W}$-gradient Method for Interval Optimization Problems}

\noindent In this section, we develop a particular type of efficient-direction. Towards this, for two given numbers $w$, $w'\in [0, 1]$ with $w+w'=1$, we define a mapping $\mathcal{W}: I(\mathbb{R})^n\rightarrow \mathbb{R}^n$ by
\[
\mathcal{W}(\textbf{A}_1, \textbf{A}_2, \ldots, \textbf{A}_n)=(w\underline{a}_1+w'\overline{a}_1, w\underline{a}_2+w'\overline{a}_2, \ldots, w\underline{a}_n+w'\overline{a}_n)^T.
\]

\begin{rmrk}\label{rw1}
It is to observe that for any two elements $\bar{\textbf{A}}$, $\bar{\textbf{B}}$ in $I(\mathbb{R})^n$,
\[
\mathcal{W}(\bar{\textbf{A}}\oplus\bar{\textbf{B}})=\mathcal{W}(\bar{\textbf{A}})+\mathcal{W}(\bar{\textbf{B}}).
\]
\end{rmrk}
\begin{lem}\label{lw1}
For an interval $\textbf{A}=[\underline{a}, \overline{a}]$, $(w\underline{a}+w'\overline{a})\odot[\underline{a}, \overline{a}]~\nprec~ \textbf{0}$, where $w$, $w'\in [0, 1]$ with $w+w'=1$.
\end{lem}
\begin{proof}
For an interval $\textbf{A}=[\underline{a}, \overline{a}]$, we have the following three cases.
\begin{enumerate}[$\bullet$ \textbf{Case} 1.]
\item If \emph{$\underline{a}~\geq~ 0$ and $\overline{a}~\geq~ 0$}. Then, $\underline{a}(w\underline{a}+w'\overline{a})~\geq~ 0$ and $\overline{a}(w\underline{a}+w'\overline{a})~\geq~ 0$. Hence,
    \[
    \textbf{0}\preceq(w\underline{a}+w'\overline{a})\odot[\underline{a}, \overline{a}]~\Longrightarrow~ (w\underline{a}+w'\overline{a})\odot[\underline{a}, \overline{a}]~\nprec~ \textbf{0}.
    \]
\item If \emph{$\underline{a}~<~ 0$ and $\overline{a}~\leq~ 0$}. Then, $\underline{a}(w\underline{a}+w'\overline{a})~\geq~ 0$ and $\overline{a}(w\underline{a}+w'\overline{a})~\geq~ 0$. Hence,
    \[
    \textbf{0}\preceq(w\underline{a}+w'\overline{a})\odot[\underline{a}, \overline{a}]~\Longrightarrow~ (w\underline{a}+w'\overline{a})\odot[\underline{a}, \overline{a}]~\nprec~ \textbf{0}.
    \]
\item If \emph{$\underline{a}~\leq~ 0$ and $\overline{a}~>~ 0$}. Then, either $(w\underline{a}+w'\overline{a})=0$ or $(w\underline{a}+w'\overline{a})~\neq~ 0$.
    If $(w\underline{a}+w'\overline{a})=0$, we have
    \[
    (w\underline{a}+w'\overline{a})\odot[\underline{a}, \overline{a}]=\textbf{0}~\Longrightarrow~ (w\underline{a}+w'\overline{a})\odot[\underline{a}, \overline{a}]~\nprec~ \textbf{0}.
    \]
    Further, if $(w\underline{a}+w'\overline{a})~\neq~ 0$, the terms $\underline{a}(w\underline{a}+w'\overline{a})$ and $\overline{a}(w\underline{a}+w'\overline{a})$ are alternative in sign. So,
    \[
    (w\underline{a}+w'\overline{a})\odot[\underline{a}, \overline{a}]~\nprec~ \textbf{0}.
    \]
\end{enumerate}
\end{proof}
\begin{thm}\label{tndd3}
Let $\mathcal{X}\subseteq\mathbb{R}^n$ be a nonempty set and $\textbf{F}(x):\mathcal{X}\rightarrow I(\mathbb{R})$ be $gH$-differentiable at a point $\bar{x}\in\mathcal{X}$. Then, the direction $-\mathcal{W}(\nabla\textbf{F}(\bar{x}))$ is an efficient-direction of $\textbf{F}$ at $\bar{x}$, provided $0\not\in D_i\textbf{F}(\bar{x})$ for at least one $i\in \{1, 2, \ldots, n\}$.
\end{thm}
\begin{proof}
Let $D_i\textbf{F}(\bar{x})=\textbf{A}_i$ for each $i\in \{1, 2, \ldots, n\}$ and $0\not\in \textbf{A}_i$ for at least one $i$. Thus,
\[
(\mathcal{W}(\nabla\textbf{F}(\bar{x})))^T\odot\nabla\textbf{F}(\bar{x})= \bigoplus_{i=1}^n(w\underline{a}_i+w'\overline{a}_i)\odot[\underline{a}_i, \overline{a}_i].
\]
Let $0\not\in \textbf{A}_i$ for $i=j$. Therefore, $\underline{a}_j$ and $\overline{a}_j$ both are either positive or negative. Thus, $(w\underline{a}_j+w'\overline{a}_j)~\neq~ 0$, and hence $(w\underline{a}_j+w'\overline{a}_j)\odot[\underline{a}_j, \overline{a}_j]~\neq~\textbf{0}$. Therefore, according to the property of interval addition we get
\begin{equation}\label{endd3a}
(\mathcal{W}(\nabla\textbf{F}(\bar{x})))^T\odot\nabla\textbf{F}(\bar{x})~\neq~\textbf{0}.
\end{equation}
If possible, let $(\mathcal{W}(\nabla\textbf{F}(\bar{x})))^T\odot\nabla\textbf{F}(\bar{x})~\prec~ \textbf{0}$. Hence, $(w\underline{a}_i+w'\overline{a}_i)\odot[\underline{a}_i, \overline{a}_i]~\prec~ \textbf{0}$ for at least one $i\in \{1, 2, \ldots, n\}$, which is not possible due to Lemma \ref{lw1}. Therefore,
\begin{equation}\label{endd3b}
(\mathcal{W}(\nabla\textbf{F}(\bar{x})))^T\odot\nabla\textbf{F}(\bar{x})~\nprec~ \textbf{0}~\Longrightarrow~ \textbf{0}\nprec-(\mathcal{W}(\nabla\textbf{F}(\bar{x})))^T\odot\nabla\textbf{F}(\bar{x}).
\end{equation}
By relations (\ref{endd3a}) and (\ref{endd3b}), we obtain
\[
\textbf{0}\npreceq-(\mathcal{W}(\nabla\textbf{F}(\bar{x})))^T\odot\nabla\textbf{F}(\bar{x}).
\]
Hence, by Theorem \ref{tndd1}, $-\mathcal{W}(\nabla\textbf{F}(\bar{x}))$ is an efficient-direction of $\textbf{F}$ at $\bar{x}$, where the corresponding efficient point $x'$, given by
\[
x'=\bar{x}-\alpha'\mathcal{W}(\nabla\textbf{F}(\bar{x}))~~\text{with}~~\alpha'=\argeff_{\alpha\in \mathbb{R}_+}~\textbf{F}(\bar{x}-\alpha \mathcal{W}(\nabla\textbf{F}(\bar{x}))),
\]
satisfies the condition (\ref{ced1}) of Definition \ref{ded} at $\bar{x}$.
\end{proof}
Based on Theorem \ref{tndd3} and Corollary \ref{cred2}, Algorithm \ref{algoes1} is reduced to the following Algorithm \ref{algoes2}. As Algorithm \ref{algoes2} is a particular case of Algorithm \ref{algoes1}, we name the method as $\mathcal{W}$-$gH$-gradient efficient-direction method. The letter `$\mathcal{W}$' is due to the reason that we use the mapping $\mathcal{W}: I(\mathbb{R})^n\rightarrow \mathbb{R}^n$ to generate efficient-direction at each iteration in Algorithm \ref{algoes2}.
\begin{algorithm}[H]
    \caption{$\mathcal{W}$-$gH$-Gradient Method for IOP}\label{algoes2}
    \begin{algorithmic}[1]
      \REQUIRE Give $w \in [0, 1]$, the initial point $x_{0}$ and the IVF $\textbf{F}:\mathcal{X}(\subseteq\mathbb{R}^n)\rightarrow I(\mathbb{R})$.
      \STATE Set $k = 0$.
      \STATE If $0\in D_i\textbf{F}(x_{k})$ for all $i\in \{1, 2, \ldots, n\}$, then Return $x_{k}$ as an efficient-solution and Stop. Otherwise go to Step 3.
      \STATE Set $d_{k}=-\mathcal{W}(\nabla\textbf{F}(x_k))$ and find an $\alpha_{k}$, where
            \[
            \alpha_{k}=\argeff_{\alpha\in \mathbb{R}_+}~\textbf{F}(x_k-\alpha d_k).\]
       \STATE Calculate
             \[
             x_{k+1} = x_{k} + \alpha_{k}d_{k}.
             \]
      \STATE Set $k \leftarrow k + 1$ and go to Step 2.
    \end{algorithmic}
  \end{algorithm}
\begin{rmrk}
It is to be mentioned that in Algorithm \ref{algoes2}, for the degenerate case of the IVF $\textbf{F}$, i.e., for $\underline{f}(x)=\overline{f}(x)=f(x)$ for all $x\in\mathcal{X}$, the direction $d_k$ will be $-\nabla f(x_k)$ and step length $\alpha_k$ will be $\text{argmin}f(x_k+\alpha d_k)$ at each iteration $k$, for any values of $w$, $w'~\geq~ 0$ with $w+w'=1$ in the mapping $\mathcal{W}$. Thus, in that case, $\mathcal{W}$-$gH$-gradient efficient method is same as steepest descent method.\\
\end{rmrk}
\begin{rmrk}
At each iteration $k$ in Algorithm \ref{algoes2}, one may choose $\alpha_{k}$ such that
\begin{equation}\label{calpha}
            \alpha_{k}\in \left\{\alpha~|~w\underline{a}+w'\overline{a}=0, ~~\text{where}~~[\underline{a}, \overline{a}]=\textbf{F}'(x_{k}+\alpha d_{k})\right\},
\end{equation}
where $\textbf{F}'(x_{k}+\alpha d_{k})$ is the $gH$-derivative of $\textbf{F}(x_{k}+\alpha d_{k})$ with respect to $\alpha$. But it is to be kept in mind that  (\ref{calpha})  is a necessary
condition to be $\alpha_k$ an argeff of $\textbf{F}(x_k-\alpha d_k)$, not sufficient.

\end{rmrk}
\begin{rmrk}
One question may arise: for a given pair of non-negative $w$ and $w'$ with $w+w'=1$, if we consider the real-valued function $w\underline{f}(x)+w'\overline{f}(x)$ corresponding to the IOP (\ref{IOP}) and apply the conventional steepest descent method, whether the obtained direction and step length in each iteration are identical with those obtained in each iteration of $\mathcal{W}$-$gH$-gradient efficient method?\\

\noindent The answer is \emph{yes} only when
\[
D_i\textbf{F}(x)=\left[\frac{\partial\underline{f}(x)}{\partial x_i}, \frac{\partial\overline{f}(x)}{\partial x_i}\right]~\text{for all}~i\in\{1,\ 2,\ \ldots,\ n\}~\text{at each}~x\in\mathcal{X},
\]
which is not true in general (see \cite{chalco2013calculus} for details).
\end{rmrk}
\begin{lem}\label{lw2}
For two elements $\bar{\textbf{A}}$ and $\bar{\textbf{B}}$ in $I(\mathbb{R})^n$ and a vector $h\in\mathbb{R}^n$, 
\[
\textbf{0}~\preceq~ h^T\odot\left(\bar{\textbf{A}}\ominus_{gH}\bar{\textbf{B}}\right)~\Longrightarrow~ \left(\mathcal{W}\left(\bar{\textbf{A}}\right)-\mathcal{W}\left(\bar{\textbf{B}}\right)\right)^T h~\geq~ 0.
\]
\end{lem}
\begin{proof}
\begingroup\allowdisplaybreaks\begin{align*}
h^T\odot\left(\bar{\textbf{A}}\ominus_{gH}\bar{\textbf{B}}\right) &=h^T\odot\left(\left(\textbf{A}_1, \textbf{A}_2,\ldots,\textbf{A}_n\right)^T\ominus_{gH}\left(\textbf{B}_1, \textbf{B}_2,\ldots,\textbf{B}_n\right)^T\right)\\
&=h^T\odot\left(\left(\textbf{A}_1\ominus_{gH}\textbf{B}_1\right), \left(\textbf{A}_2\ominus_{gH}\textbf{B}_2\right),\ldots,\left(\textbf{A}_n\ominus_{gH}\textbf{B}_n\right)\right)^T\\
&=\bigoplus_{i=1}^n h_i\odot\left(\textbf{A}_i\ominus_{gH}\textbf{B}_i\right).
\end{align*}\endgroup
According to Remark \ref{ria1}, without loss of generality, let us assume
\begingroup\allowdisplaybreaks\begin{align*}
\bigoplus_{i=1}^n h_i\odot\left(\textbf{A}_i\ominus_{gH}\textbf{B}_i\right) =~ &\bigoplus_{k=1}^p h_k\odot\left(\textbf{A}_k\ominus_{gH}\textbf{B}_k\right) \oplus\bigoplus_{l=p+1}^n h_l\odot\left(\textbf{A}_l\ominus_{gH}\textbf{B}_l\right)\\
=~ & \bigoplus_{k=1}^p \textbf{C}_k \oplus\bigoplus_{l=p+1}^n \textbf{C}_l,
\end{align*}\endgroup
where $[\underline{c}_k, \overline{c}_k]= h_k\odot\left(\textbf{A}_k\ominus_{gH}\textbf{B}_k\right) =\left[\left(\underline{a}_k-\underline{b}_k\right)h_k, \left(\overline{a}_k-\overline{b}_k\right)h_k\right]
$ for all $k$'s and $[\underline{c}_l, \overline{c}_l]=h_l\odot\left(\textbf{A}_l\ominus_{gH}\textbf{B}_l\right) =\left[\left(\overline{a}_l-\overline{b}_l\right)h_l, \left(\underline{a}_l-\underline{b}_l\right)h_l\right]
$ for all $l$'s. Thus,
\begingroup\allowdisplaybreaks\begin{align*}
\textbf{0}~\preceq~ h^T\odot\left(\bar{\textbf{A}}\ominus_{gH}\bar{\textbf{B}}\right)
~\Longrightarrow~ & \textbf{0}~\preceq~ \bigoplus_{k=1}^p[\underline{c}_k, \overline{c}_k]  \oplus\bigoplus_{l=p+1}^n [\underline{c}_l, \overline{c}_l]\\
~\Longrightarrow~ & \textbf{0}~\preceq~ \left[\sum_{k=1}^p\underline{c}_k +\sum_{l=p+1}^n\underline{c}_l, \sum_{k=1}^p\overline{c}_k +\sum_{l=p+1}^n\overline{c}_l\right]\\
~\Longrightarrow~ & 0~\leq~ \sum_{k=1}^p\underline{c}_k +\sum_{l=p+1}^n\underline{c}_l~\text{and}~ 0~\leq~ \sum_{k=1}^p\overline{c}_k +\sum_{l=p+1}^n\overline{c}_l.
\end{align*}\endgroup
Therefore,
\begingroup\allowdisplaybreaks\begin{align*}
& \sum_{k=1}^p\underline{c}_k +\sum_{l=p+1}^n\overline{c}_l \geq \sum_{k=1}^p\underline{c}_k +\sum_{l=p+1}^n\underline{c}_l\geq 0 ~\text{and}~ \sum_{k=1}^p\overline{c}_k +\sum_{l=p+1}^n\underline{c}_l \geq \sum_{k=1}^p\underline{c}_k +\sum_{l=p+1}^n\underline{c}_l \geq 0\\
\Longrightarrow & \sum_{k=1}^p\left(\underline{a}_k-\underline{b}_k\right)h_k +\sum_{l=p+1}^n\left(\underline{a}_l-\underline{b}_l\right)h_l ~\geq~ 0 ~\text{and}~ \sum_{k=1}^p\left(\overline{a}_k-\overline{b}_k\right)h_k +\sum_{l=p+1}^n\left(\overline{a}_l-\overline{b}_l\right)h_l ~\geq~ 0\\
\Longrightarrow & \sum_{i=1}^n\left(\underline{a}_i-\underline{b}_i\right)h_i ~\geq~ 0 ~\text{and}~ \sum_{i=1}^n\left(\overline{a}_i-\overline{b}_i\right)h_i ~\geq~ 0 \\
\Longrightarrow & w\sum_{i=1}^n\left(\underline{a}_i-\underline{b}_i\right)h_i ~\geq~ 0 ~\text{and}~ w'\sum_{i=1}^n\left(\overline{a}_i-\overline{b}_i\right)h_i ~\geq~ 0 \\
\Longrightarrow & \sum_{i=1}^nw\left(\underline{a}_i-\underline{b}_i\right)h_i +w'\left(\overline{a}_i-\overline{b}_i\right)h_i ~\geq~ 0 \\
\Longrightarrow &\left(\mathcal{W}\left(\bar{\textbf{A}}\right)-\mathcal{W}\left(\bar{\textbf{B}}\right)\right)^Th~\geq~ 0.
\end{align*}\endgroup
\end{proof}
\begin{lem}{\label{lw3}}

If $\textbf{F}$ is $gH$-differentiable IVF, then

\begin{equation}{\label{qqq}}
\rVert \mathcal{W}(\nabla \textbf{F}(x))-\mathcal{W}(\nabla \textbf{F}(y)) \rVert_{\mathcal{X}} ~\leq~ \rVert \nabla \textbf{F}(x)\ominus_{gH}\nabla \textbf{F}(y) \rVert_{I(\mathbb{R})^n} ~~\text{for all}~~ x,~ y \in \mathcal{X}.
\end{equation}

\end{lem}
\begin{proof} Let
\[
\nabla \textbf{F}(x)=\bar{\textbf{A}}=\left(\textbf{A}_1, \textbf{A}_2,\ldots,\textbf{A}_n\right)^T=\left([\underline{a}_1, \overline{a}_1], [\underline{a}_2, \overline{a}_2],\ldots,[\underline{a}_n, \overline{a}_n]\right)^T
\]
and
\[
\nabla \textbf{F}(y)=\bar{\textbf{B}}=\left(\textbf{B}_1, \textbf{B}_2,\ldots,\textbf{B}_n\right)^T=\left([\underline{b}_1, \overline{b}_1], [\underline{b}_2, \overline{b}_2],\ldots,[\underline{b}_n, \overline{b}_n]\right)^T.
\]
Therefore,
\begingroup\allowdisplaybreaks\begin{align*}
\rVert \nabla \textbf{F}(x)\ominus_{gH}\nabla \textbf{F}(y)\rVert_{I(\mathbb{R})^n}&~=~\rVert (\textbf{A}_1\ominus_{gH} \textbf{B}_1, \textbf{A}_2\ominus_{gH} \textbf{B}_2, \ldots ,\textbf{A}_n\ominus_{gH} \textbf{B}_n )\rVert_{I(\mathbb{R})^n}\\
&~=~\rVert \textbf{A}_1\ominus_{gH} \textbf{B}_1\rVert_{I(\mathbb{R})}+\rVert \textbf{A}_2\ominus_{gH} \textbf{B}_2\rVert_{I(\mathbb{R})}+\cdots+\rVert\textbf{A}_n\ominus_{gH} \textbf{B}_n \rVert_{I(\mathbb{R})}.
\end{align*}\endgroup
We note that
\[
\mathcal{W}(\nabla \textbf{F}(x))=(w\underline{a}_1+w'\overline{a}_1, w\underline{a}_2+w'\overline{a}_2,\ldots, w\underline{a}_n+w'\overline{a}_n)^T
\]
and
\[
\mathcal{W}(\nabla \textbf{F}(y))=(w\underline{b}_1+w'\overline{b}_1, w\underline{b}_2+w'\overline{b}_2,\ldots, w\underline{b}_n+w'\overline{b}_n)^T.
\]
So,
\begingroup\allowdisplaybreaks\begin{align*}
&\rVert\mathcal{W}(\nabla \textbf{F}(x))-\mathcal{W}(\nabla \textbf{F}(y))\rVert_{\mathcal{X}}\\
~=~&\rVert(w(\underline{a}_1-\underline{b}_1)+w'(\overline{a}_1-\overline{b}_1), w(\underline{a}_2-\underline{b}_2)+w'(\overline{a}_2-\overline{b}_2),\\
&~~~~~~~~~~~~\cdots, w(\underline{a}_n-\underline{b}_n)+w'(\overline{a}_n-\overline{b}_n))\rVert_{\mathcal{X}}\\
~\leq~&\lvert w(\underline{a}_1-\underline{b}_1)+w'(\overline{a}_1-\overline{b}_1)\rvert+\lvert w(\underline{a}_2-\underline{b}_2)+w'(\overline{a}_2-\overline{b}_2)\rvert\\
     &~~~~~~~~+\cdots+\lvert w(\underline{a}_n-\underline{b}_n)+w'(\overline{a}_n-\overline{b}_n\rvert \\
~\leq~&\max \{\lvert \underline{a}_1-\underline{b}_1 \rvert, \lvert \overline{a}_1-\overline{b}_1\rvert \}+\max \{\lvert \underline{a}_2-\underline{b}_2 \rvert, \lvert \overline{a}_2-\overline{b}_2\rvert \}\\
		&~~~~~~~~+\cdots+\max \{\lvert \underline{a}_n-\underline{b}_n \rvert, \lvert \overline{a}_n-\overline{b}_n\rvert \}\\
~=~&\rVert \textbf{A}_1\ominus_{gH} \textbf{B}_1\rVert_{I(\mathbb{R})}+\rVert \textbf{A}_2\ominus_{gH} \textbf{B}_2\rVert_{I(\mathbb{R})}+\cdots+\rVert\textbf{A}_n\ominus_{gH} \textbf{B}_n \rVert_{I(\mathbb{R})}\\
~=~&\rVert \nabla \textbf{F}(x)\ominus_{gH}\nabla \textbf{F}(y) \rVert_{I(\mathbb{R})^\text{n}}.
\end{align*}\endgroup
\end{proof}
\begin{lem}\label{lw4}
Let $\textbf{F}$ be $gH$-differentiable and a strongly convex IVF on a nonempty convex subset $\mathcal{X}$ of $\mathbb{R}^n$. Then, there exists a $\sigma>0$ such that
\[
\left(\mathcal{W}(\nabla \textbf{F}(x))-\mathcal{W}(\nabla \textbf{F}(y))\right)^T(x-y) ~\geq~ \sigma \rVert x-y \rVert^2 ~~\text{for all}~~ x,~ y \in \mathcal{X}.
\]	
\end{lem}
\begin{proof}
Since $\textbf{F}$ is strong convex on $\mathcal{X}$, then there exists a convex IVF $\textbf{G}$ on $\mathcal{X}$ such that
\[
\textbf{G}(x)\oplus \frac{1}{2}\rVert x \rVert^2\odot[\sigma, \sigma]=\textbf{F}(x)~~\text{for all}~~ x\in \mathcal{X}.
\]
where $\sigma~>~0$. Therefore,
\begingroup\allowdisplaybreaks\begin{align*}
&\nabla \textbf{G}(x)\oplus ([\sigma x_1, \sigma x_1], [\sigma x_2, \sigma x_2], \ldots, [\sigma x_n, \sigma x_n])^T =\nabla \textbf{F}(x)\\
\Longleftrightarrow~&\mathcal{W}(\nabla \textbf{G}(x))+ \sigma x =\mathcal{W}(\nabla \textbf{F}(x)),~~ \text{by Remark \ref{rw1}}.
\end{align*}\endgroup
This implies
\begin{equation}{\label{hhhh}}
\mathcal{W}(\nabla \textbf{G}(x)) =\mathcal{W}(\nabla \textbf{F}(x))- \sigma x.
\end{equation}
Further, since $\textbf{F}$ is $gH$-differentiable, $\textbf{G}$ is $gH$-differentiable; as $\textbf{G}$ is also convex IVF, by Lemma \ref{lw2} and Theorem \ref{td3}, we have
\[
\left(\mathcal{W}(\nabla \textbf{G}(x))-\mathcal{W}(\nabla \textbf{G}(y))\right)^T(x-y) ~\geq~ 0 ~~\text{for all}~~ x,~ y \in \mathcal{X}.
\]
Then, from (\ref{hhhh}) we obtain
\begingroup\allowdisplaybreaks\begin{align*}
&\left(\mathcal{W}(\nabla \textbf{F}(x))-\mathcal{W}(\nabla \textbf{F}(y))-\sigma (x-y)\right)^T(x-y) ~\geq~ 0\\
~\Longrightarrow~& \left(\mathcal{W}(\nabla \textbf{F}(x))-\mathcal{W}(\nabla \textbf{F}(y))\right)^T(x-y)-\sigma \rVert x-y \rVert^2 ~\geq~ 0\\
~\Longrightarrow~& \left(\mathcal{W}(\nabla \textbf{F}(x))-\mathcal{W}(\nabla \textbf{F}(y))\right)^T(x-y)~\geq~ \sigma \rVert x-y \rVert^2.
\end{align*}\endgroup
\end{proof}
\begin{lem}{\label{lw5}}
If a $gH$-differentiable IVF $\textbf{F}$ on a nonempty subset $\mathcal{X}$ of $\mathbb{R}^n$ has $gH$-Lipschitz gradient, then for some $L~>~0$ we have
\[
\rVert \mathcal{W}(\nabla \textbf{F}(x))-\mathcal{W}(\nabla \textbf{F}(y)) \rVert ~\leq~ L \rVert x-y \rVert ~~\text{for all}~~ x,~ y \in \mathcal{X}
\]
\end{lem}
\begin{proof} Since $\textbf{F}$ has $gH$-Lipschitz gradient, for some $L~>~0$, we have
\[
\lVert\nabla \textbf{F}(x)\ominus_{gH}\nabla \textbf{F}(y)\rVert_{I(\mathbb{R})} ~\leq~ L \lVert x-y \rVert ~~\text{for all}~~ x,~ y \in \mathcal{X}.
\]
By Lemma \ref{lw3} and the last relation we obtain
\[
\rVert \mathcal{W}(\nabla \textbf{F}(x))-\mathcal{W}(\nabla \textbf{F}(y)) \rVert ~\leq~ L \rVert x-y \rVert ~~\text{for all}~~ x,~ y \in \mathcal{X}.
\]
\end{proof}
\begin{lem}{\label{lw6}}
If $\textbf{F}$ is a strong convex and $gH$-differentiable IVF on a nonempty convex subset $\mathcal{X}$ of $\mathbb{R}^n$ with $gH$-Lipschitz gradient, then for all  $x$, $y \in \mathcal{X}$, there exists a $\sigma~>~0$ and an $L~>~0$ such that
\[
(\mathcal{W}(\nabla \textbf{F}(x))-\mathcal{W}(\nabla \textbf{F}(y))^T(y-x) ~\geq~ \frac{\sigma}{L^2} \lVert \mathcal{W}(\nabla \textbf{F}(x))-\mathcal{W}(\nabla \textbf{F}(y) \rVert^2~~\text{for all}~~ x,~ y \in \mathcal{X}.
\]
\end{lem}
\begin{proof} Since $\textbf{F}$ has $gH$-Lipschitz gradient, due to Lemma \ref{lw5} we get
\[
\lVert \mathcal{W}(\nabla \textbf{F}(x))-\mathcal{W}(\nabla \textbf{F}(y)) \rVert ~\leq~ L \lVert x-y \rVert~~\text{for all}~~ x,~ y \in \mathcal{X}.
\]
By Lemma \ref{lw4} and the last inequality, we obtain
\[
(\mathcal{W}(\nabla \textbf{F}(x))-\mathcal{W}(\nabla \textbf{F}(y))^T(y-x) ~\geq~ \frac{\sigma}{L^2} \lVert \mathcal{W}(\nabla \textbf{F}(x))-\mathcal{W}(\nabla \textbf{F}(y) \rVert^2~~\text{for all}~~ x,~ y \in \mathcal{X}.
\]	
\end{proof}

\begin{thm}\emph{(Linear convergence under strong convexity)}.
Let $\mathcal{X}$ be a nonempty convex subset of $\mathbb{R}^n$. If $\textbf{F}$ is a strong convex and $gH$-differentiable IVF with $gH$-Lipschitz gradient on $\mathcal{X}$. Then,  there exists a $\sigma> 0$ and an $L>0$ such that if $\bar{x}\in\mathcal{X}$ be an efficient solution of the IOP (\ref{IOP}), the mapping $H_\alpha(x)=x-\alpha \mathcal{W}(\nabla \textbf{F}(x))$ with constant step size $\alpha\in\left[0, \frac{2\sigma}{L^2}\right]$ satisfies
\[
\rVert H_\alpha(x)-H_\alpha(\bar{x}) \rVert ~\leq~ {\lVert x-\bar{x} \rVert} ~\text{for all}~ x\in \mathcal{X}.
\]
\end{thm}

\begin{proof} Let $\textbf{F}$ be strong convex and $gH$-differentiable IVF with $gH$-Lipschitz gradient on $\mathcal{X}$. Therefore, for all $x,~ y \in \mathcal{X}$, we have
\begingroup\allowdisplaybreaks\begin{align*}
&\lVert H_\alpha(x)-H_\alpha(y) \rVert^2\\
~=~&\left\lVert\big(x-\alpha \mathcal{W}(\nabla \textbf{F}(x))\big)-\big(y-\alpha \mathcal{W}(\nabla \textbf{F}(y))\big) \right\rVert^2\\
~=~&\lVert x-y \rVert^2 -2\alpha (\mathcal{W}(\nabla \textbf{F}(x))-\mathcal{W}(\nabla \textbf{F}(y))^T(x-y) +\alpha^2\lVert \mathcal{W}(\nabla \textbf{F}(x))-\mathcal{W}(\nabla \textbf{F}(y)) \rVert^2\\
~\leq~ &{\lVert x-y \rVert}^2-2\alpha\sigma{\lVert x-y \rVert}^2+\alpha^2L^2{\lVert x-y \rVert}^2 \text{ by Lemma \ref{lw4} and Lemma \ref{lw6}} \\
~\leq~ & {\lVert x-y \rVert}^2 \text{ since } \alpha\in\left[0, \frac{2\sigma}{L^2}\right].
\end{align*}\endgroup
Taking $y=\bar{x}$, we have 
$
\rVert H_\alpha(x)-H_\alpha(\bar{x}) \rVert ~\leq~ {\rVert x-\bar{x} \rVert} ~~\text{for all}~~ x\in \mathcal{X}.
$
Hence, the $\mathcal{W}$-$gH$-gradient efficient method converges linearly.
\end{proof}
%
%
%

\subsection{Numerical Examples}

\noindent Here we apply the proposed algorithm of $\mathcal{W}$-$gH$-gradient efficient method on the IOPs of the following two examples and capture the efficient solutions of the IOPs.
\begin{example}\label{exgm2}
Consider the following IOP:
\begin{equation}\label{ex2iop}
\min_{x\in [-3, 7]} \textbf{F}(x)=[1, 2]\odot x^2 \oplus [-8, 0]\odot x \oplus [3, 25].
\end{equation}
\noindent The $gH$-gradient of $\textbf{F}$ is
\[
\nabla \textbf{F}(x)=[2, 4]\odot x\oplus [-8, 0]~\text{for all}~ x\in[-3, 7].
\]
The graph of the IVF $\textbf{F}$ is illustrated in the Figure \ref{fgm2} by the gray shaded region and the region of the efficient solutions is marked by bold black line on $x$-axis. From Figure \ref{fgm2} it is clear that each $x \in [0, 4]$ is an efficient solution of the IOP (\ref{ex2iop}).
\begin{figure}[H]
\begin{center}
\includegraphics[scale=0.8]{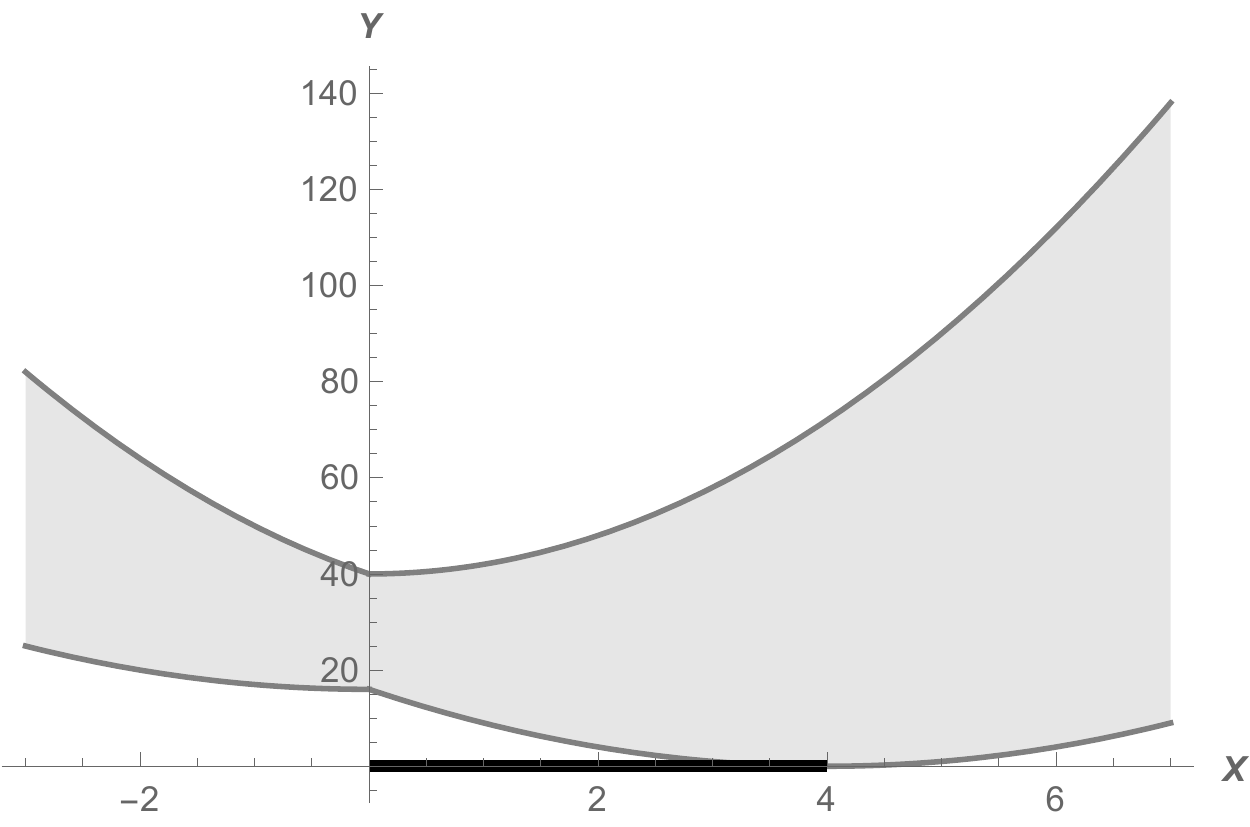}
    \caption{Interval-valued function and efficient solution of the IOP (\ref{ex2iop}) in Example \ref{exgm2}}\label{fgm2}
\end{center}
\end{figure}
Since the IOP (\ref{ex2iop}) has infinite number of efficient solutions, applying $\mathcal{W}$-$gH$-gradient efficient method for various $w$ and $w'$, we obtain few of them, which are presented in Table \ref{tablegm-1}. \\ \\
\begin{longtable}{lllll}
\caption{Output of Algorithm \ref{algoes2} to find efficient solutions of IOP (\ref{ex2iop})}\label{tablegm-1} \\
\toprule
$\multirow{2}{*} \emph{w} $ & $\multirow{2}{*} \emph{w'} $ & Initial & Number of & Efficient \\
 & & point & iterations & solution \\
\midrule
\endhead
\multirow{3}{*} {$0$} & \multirow{3}{*} {$1$} & $-2$ & $22$ & {$0$} \\
              & & $-0.5$ & $20$ & {$0$} \\
              & & $6$ & $1$ & {$3$} \\
\midrule
\multirow{3}{*} {$0.2$} & \multirow{3}{*} {$0.8$} & $-2$ & $3$ &  {$0.24$} \\
              & & $-0.5$ & $2$ & {$0.2719$} \\
              & & $6$ & $1$ & {$1.9778$} \\
\midrule
\multirow{3}{*} {$0.4$} & \multirow{3}{*} {$0.6$} & $-2$ & $2$ & {$0.2621$} \\
              & & $-0.5$ & $1$ & {$0.1549$} \\
              & & $6$ & $1$ & {$0.6875$} \\
\midrule
\multirow{3}{*} {$0.5$} & \multirow{3}{*} {$0.5$} & $-2$ & $1$ & {$0.1333$} \\
              & & $-0.5$ & $1$ & {$0.2424$} \\
              & & $6$ & $2$ & {$0.3111$} \\
\midrule
\multirow{3}{*} {$0.7$} & \multirow{3}{*} {$0.3$} & $-2$ & $1$ & {$0.7891$} \\
              & & $-0.5$ & $1$ & {$0.4489$} \\
              & & $6$ & $2$ & {$0.8051$} \\
\midrule
\multirow{3}{*} {$0.9$} & \multirow{3}{*} {$0.1$} & $-2$ & $1$ & {$2.25$} \\
              & & $-0.5$ & $1$ & {$0.7233$} \\
              & & $6$ & $2$ & {$3.6968$} \\
\bottomrule
\end{longtable}

\end{example}
As the IOP (\ref{ex2iop}) in Example \ref{exgm2} has infinite number of efficient solutions, Table \ref{tablegm-1} shows that the algorithm of $\mathcal{W}$-$gH$-gradient efficient method has stopped with different efficient solutions for different combinations of $w$, $w'$ and initial points.\\

\noindent In the next example, we consider an IOP which has only one efficient solution and we show that the algorithm of $\mathcal{W}$-$gH$-gradient efficient method will stop at the efficient solution for different combinations of $w$, $w'$ and initial points.
\begin{example}\label{exgm3}
Consider the following IOP:
\begin{equation}\label{ex3iop}
\min_{x\in \mathcal{X}\subseteq\mathbb{R}^2} \textbf{F}(x)=[2, 6]\odot (x_1-2)^2 \oplus [5, 7]\odot (x_2-3)^2 \oplus [5, 12],
\end{equation}
where $\mathcal{X}=[0, 6]\times [0, 6]$. The $gH$-gradient of $\textbf{F}$ is
\begingroup\allowdisplaybreaks\begin{align*}
\nabla \textbf{F}(x)&=(D_1\textbf{F}(x), D_2\textbf{F}(x))^T\\
&=\left ([4, 12]\odot (x_1-2),~[10,14]\odot (x_2-3)\right )^T.
\end{align*}\endgroup

\noindent We show that $\bar{x}=(2, 3)^T$ is an efficient solution to IOP (\ref{ex3iop}). On contrary, if possible, let there exist two nonzero numbers $h_1$ and $h_2$ with $0 ~\leq~ 2+h_1 ~\leq~ 6$ and $0 ~\leq~ 3+h_2 ~\leq~ 6$ such that
\begingroup\allowdisplaybreaks\begin{align*}
& \textbf{F}(2+h_1, 3+h_2) ~\prec~ \textbf{F}(2, 3)\\
\text{or, } & [2, 6]\odot h_1^2 \oplus [5, 7]\odot h_2^2 \oplus [5, 12] ~\prec~ [5, 12]\\
\text{or, } & \left([2, 6]\odot h_1^2 \oplus [5, 7]\odot h_2^2 \oplus [5, 12]\right) \ominus_{gH} [5, 12] ~\prec~ [0, 0]\\
\text{or, } & [2h_1^2+5h_2^2, 6h_1^2+7h_2^2] ~\prec~ [0, 0],
\end{align*}\endgroup
which is not possible. Thus, there does not exist any $x (\neq\bar{x})\in\mathcal{X}$ such that $\textbf{F}(x) ~\prec~ \textbf{F}(\bar{x})$. Hence, $\bar{x}$ is an efficient solution of the IOP (\ref{ex3iop}). Also, one can easily check that $\textbf{F}(\bar{x}) ~\prec~ \textbf{F}(x)$. Thus, $\bar{x}$ is the only efficient solution of the IOP (\ref{ex3iop}).\\

\noindent The IVF $\textbf{F}$ is depicted in the Figure \ref{fgm3} by the gray shaded surface and the efficient solution is pointed by black dot on $xy$-plane. The Figure \ref{fgm3} shows that $\bar{x}$ is the only efficient solution of $\textbf{F}$ on $\mathcal{X}$.
\begin{figure}[H]
\begin{center}
\includegraphics[scale=0.7]{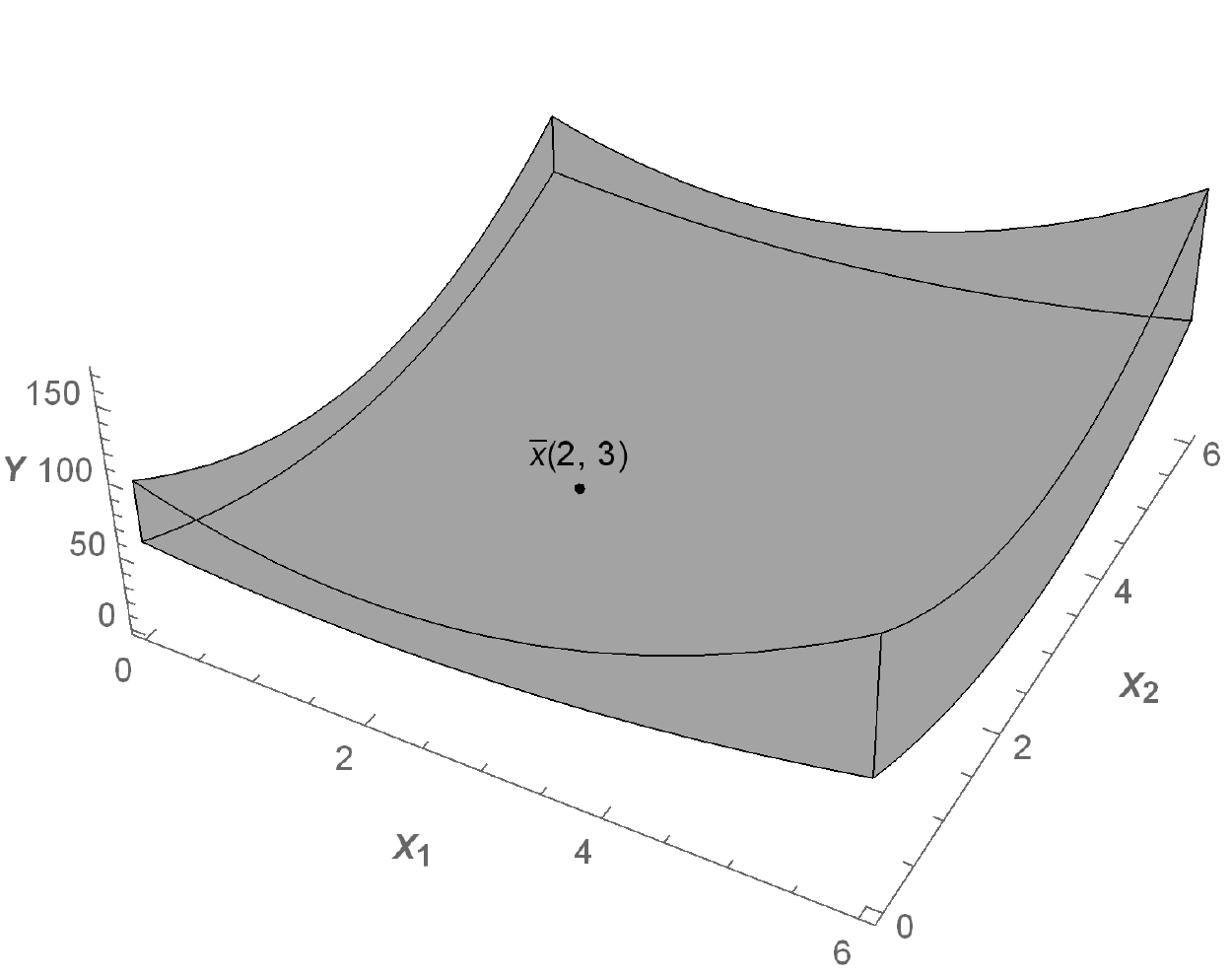}
    \caption{Interval-valued function and efficient solution of the IOP (\ref{ex3iop}) in Example \ref{exgm3}}\label{fgm3}
\end{center}
\end{figure}
From Table \ref{tablegm-2}, we see that for every combination of $w$, $w'$ and initial point, the $\mathcal{W}$-$gH$-gradient efficient method has stopped at the efficient solution $\bar{x}=(2, 3)^T$. 
\begin{longtable}{lllll}
\caption{Output of Algorithm \ref{algoes2} to find efficient solutions of IOP (\ref{ex3iop})}\label{tablegm-2} \\
\toprule
$\multirow{2}{*} \emph{w} $ & $\multirow{2}{*} \emph{w'} $ & Initial & Number of & Efficient \\
 & & point & iterations & solution \\
\midrule
\endhead
\multirow{3}{*} {$0.1$} & \multirow{3}{*} {$0.9$} & $(0, 6)^T$ & $30$ & \multirow{3}{*} {$(2, 3)^T$}\\
               &  & $(5, 2)^T$ & $21$ &  \\
               &  & $(2.5, 2.5)^T$ & $13$ &  \\
\midrule
\multirow{3}{*} {$0.3$} & \multirow{3}{*} {$0.7$} & $(0, 6)^T$ & $16$ & \multirow{3}{*} {$(2, 3)^T$}\\
               &  & $(5, 2)^T$ & $15$ &  \\
               &  & $(2.5, 2.5)^T$ & $13$ &  \\
\midrule
\multirow{3}{*} {$0.4$} & \multirow{3}{*} {$0.6$} & $(0, 6)^T$ & $12$ & \multirow{3}{*} {$(2, 3)^T$}\\
               &  & $(5, 2)^T$ & $10$ &  \\
               &  & $(2.5, 2.5)^T$ & $9$ &  \\
\multirow{3}{*} {$0.6$} & \multirow{3}{*} {$0.4$} & $(0, 6)^T$ & $10$ & \multirow{3}{*} {$(2, 3)^T$}\\
               &  & $(5, 2)^T$ & $9$ &  \\
               &  & $(2.5, 2.5)^T$ & $11$ &  \\
\midrule
\multirow{3}{*} {$0.9$} & \multirow{3}{*} {$0.1$} & $(0, 6)^T$ & $20$ & \multirow{3}{*} {$(2, 3)^T$}\\
               &  & $(5, 2)^T$ & $19$ &  \\
               &  & $(2.5, 2.5)^T$ & $18$ &  \\
\midrule
\multirow{3}{*} {$1$} & \multirow{3}{*} {$0$} & $(0, 6)^T$ & $25$ & \multirow{3}{*} {$(2, 3)^T$}\\
               &  & $(5, 2)^T$ & $21$ &  \\
               &  & $(2.5, 2.5)^T$ & $22$ &  \\
\bottomrule
\end{longtable}
\end{example}

\section{Application}\label{sa}

\noindent \hl{This section applies the $\mathcal{W}$-$gH$-gradient efficient method in solving the \emph{least square problems} for interval-valued data.}\\

Suppose a set of  $n$ pairs of data $(\textbf{X}_k, \textbf{Y}_k)$ is given, where $\textbf{Y}_k\in I(\mathbb{R})$ is the corresponding interval-valued output of $\textbf{X}_k\in I(\mathbb{R})^p$ for all $k\in\{1, 2, \ldots, n\}$. We attempt to fit a function $\boldsymbol{\mathcal{H}}\left(\cdot ;\beta\right):I(\mathbb{R})^p\to I(\mathbb{R})$, where $\beta\in \mathbb{R}^l$ is a parameter vector such that $\boldsymbol{\mathcal{H}}\left(\textbf{X}_k;\hat{\beta}\right)$ will be one of the best approximations of $\textbf{Y}_k$ for all $k\in\{1, 2, \ldots, n\}$. By `one of the best approximations' we mean that $\boldsymbol{\mathcal{H}}\left(\textbf{X} ;\beta\right)$ gives a sum square error that is nondominated. Evidently, if $\hat{\beta}$ is an efficient solution of the following IOP:
\begin{equation}\label{eef}
\min_{\beta\in\mathbb{R}^l} {\color{red}\boldsymbol{\mathcal{E}}(\beta)}=\bigoplus_{k=1}^n\left(\boldsymbol{\mathcal{H}}\left(\textbf{X}_k;\beta\right)\ominus_{gH}\textbf{Y}_k\right) \odot\left(\boldsymbol{\mathcal{H}}\left(\textbf{X}_k;\beta\right)\ominus_{gH}\textbf{Y}_k\right),
\end{equation}
then $\boldsymbol{\mathcal{H}}\left(\textbf{X};\hat{\beta}\right)$ can be considered as an efficient choice of the approximating function $\boldsymbol{\mathcal{H}}\left(\textbf{X};\beta\right)$.\\

\noindent It is noteworthy that the error function ${\color{red}\boldsymbol{\mathcal{E}}}$ and the function $\boldsymbol{\mathcal{H}}\left(\textbf{X}_k; \cdot\right)$ are IVFs from $\mathbb{R}^l$ to $I(\mathbb{R})$ for all $\textbf{X}_k$. The partial $gH$-derivative  of ${\color{red}\boldsymbol{\mathcal{E}}}$ with respect to $\beta_i$ is
\[
D_i {\color{red}\boldsymbol{\mathcal{E}}(\beta)}=2\odot\bigoplus_{k=1}^n\left(\boldsymbol{\mathcal{H}}\left(\textbf{X}_k;\beta\right)\ominus_{gH}\textbf{Y}_k\right)\odot D_i \boldsymbol{\mathcal{H}}\left(\textbf{X}_k;\beta\right)~\text{ for all }~i\in\{1, 2, \ldots, l\}.
\]
Hence, by applying the $\mathcal{W}$-$gH$-gradient efficient method on the IOP (\ref{eef}) one can easily obtain an efficient parameter vector $\hat{\beta}$ for the function $\boldsymbol{\mathcal{H}}\left(\cdot; \beta\right)$. For examples, we consider the following two types of fitting with interval-valued data.\\

\subsection{Polynomial Fitting}

\noindent Let us consider a set of $21$ pairs of interval-valued data that are displayed in Table \ref{tabledata-1}. We attempt fit a polynomial function $\boldsymbol{\mathcal{H}}^1\left(\cdot ;\beta\right):I(\mathbb{R})\to I(\mathbb{R})$, defined by
    \[
    \boldsymbol{\mathcal{H}}^1\left(\textbf{X};\beta\right)=\beta_1\odot \textbf{C} \oplus \beta_2 \odot \textbf{X} \oplus \beta_3 \odot \textbf{X}^2,
    \]
where $\textbf{C}$ is a constant interval and $\beta=(\beta_1, \beta_2, \beta_3) \in \mathbb{R}^3$. Therefore, for each $k=1,\ 2,\ \ldots, \ 21$, the partial $gH$-derivative s of $\boldsymbol{\mathcal{H}}^1\left(\textbf{X}_k; \cdot \right)$ with respect to $\beta_1$, $\beta_2$ and $\beta_3$ are
\[
D_1 \boldsymbol{\mathcal{H}}^1\left(\textbf{X}_k; \beta \right) = \textbf{C},~~D_2 \boldsymbol{\mathcal{H}}^1\left(\textbf{X}_k; \beta \right) = \textbf{X}_k ~\text{ and }~D_3 \boldsymbol{\mathcal{H}}\left(\textbf{X}_k; \beta \right) = \textbf{X}_k^2, \text{ respectively}.
\]
\begin{table}[H]
\caption{Data for polynomial fitting}\label{tabledata-1}
\begin{center}
\begin{tabular}{llll}
\toprule
$\textbf{X}_k$ & $\textbf{Y}_k$ & $\textbf{X}_k$ & $\textbf{Y}_k$\\
\midrule
$[-1.99, -1.86]$ & $[1.21, 2.60]$ & $[-1.80, -1.66]$ & $[0.72, 2.00]$ \\

$[-1.61, -1.54]$ & $[0.43, 1.67]$ & $[-1.48, -1.33]$ & $[0.20, 1.39]$ \\

$[-1.29, -1.14]$ & $[0.01, 1.22]$ & $[-1.08, -0.96]$ & $[-0.28, 1.00]$ \\

$[-0.89, -0.73]$ & $[-0.60, 0.60]$ & $[-0.68, -0.51]$ & $[-0.8, 0.34]$ \\

$[-0.45, -0.29]$ & $[-0.95, 0.12]$ & $[-0.23, -0.03]$ & $[-1.12, -0.10]$ \\

$[0.01, 0.15]$ & $[-1.24, 0.00]$ & $[0.19, 0.33]$ & $[-1.3, 0.01]$ \\

$[0.39, 0.54]$ & $[-1.32, 0.02]$ & $[0.60, 0.74]$ & $[-1.25, 0.10]$ \\

$[0.79, 0.93]$ & $[-1.05, 0.13]$ & $[0.98, 1.13]$ & $[-0.96, 0.20]$ \\

$[1.19, 1.33]$ & $[-0.69, 0.40]$ & $[1.39, 1.54]$ & $[-0.34, 0.78]$ \\

$[1.60, 1.74]$ & $[-0.03, 1.09]$ & $[1.79, 1.95]$ & $[0.24, 1.47]$ \\

$[2.00, 2.15]$ & $[0.55, 1.80]$ \\
\bottomrule
\end{tabular}
\end{center}
\end{table}

Considering $\textbf{C}=[1.70, 12.00]$ and an initial $\beta = (6, -8, 9)^T$, and applying Algorithm \ref{algoes2} on the IOP (\ref{eef}) corresponding to the function $\boldsymbol{\mathcal{H}}^1\left(\cdot ;\beta\right)$ with $w = 0.3$, we obtain the value of $\hat{\beta}$ equal to $(-0.0876, -0.2974, 0.5458)^T$ in $24$ iterations. With $w = 0.5$, we obtain the value of $\hat{\beta}$ equal to $(-0.0896, -0.2777, 0.5352)^T$ in $10$ iterations. \\

In both the figures of Figure \ref{fpf} show the comparison of the actual interval-valued output $\textbf{Y}_k$ with the estimated output $\boldsymbol{\mathcal{H}}^1\left(\textbf{X}_k ;\hat{\beta}\right)$ of the interval-valued data $\textbf{X}_k$ in polynomial fitting for the values of $w$ equal to $0.3$ and $0.5$, respectively, for $k = 1,\ 2,\ \ldots, \ 21$. The common portions of $\textbf{Y}_k$ with $\boldsymbol{\mathcal{H}}^1\left(\textbf{X}_k;\hat{\beta}\right)$ are depicted by orange regions, where as the extended portions of $\textbf{Y}_k$ and $\boldsymbol{\mathcal{H}}^1\left(\textbf{X}_k ;\hat{\beta}\right)$ are illustrated by red and yellow regions, respectively.

\begin{figure}[H]
\begin{center}
\begin{subfigure}[For $w=0.3$]
    {\includegraphics[scale=0.26]{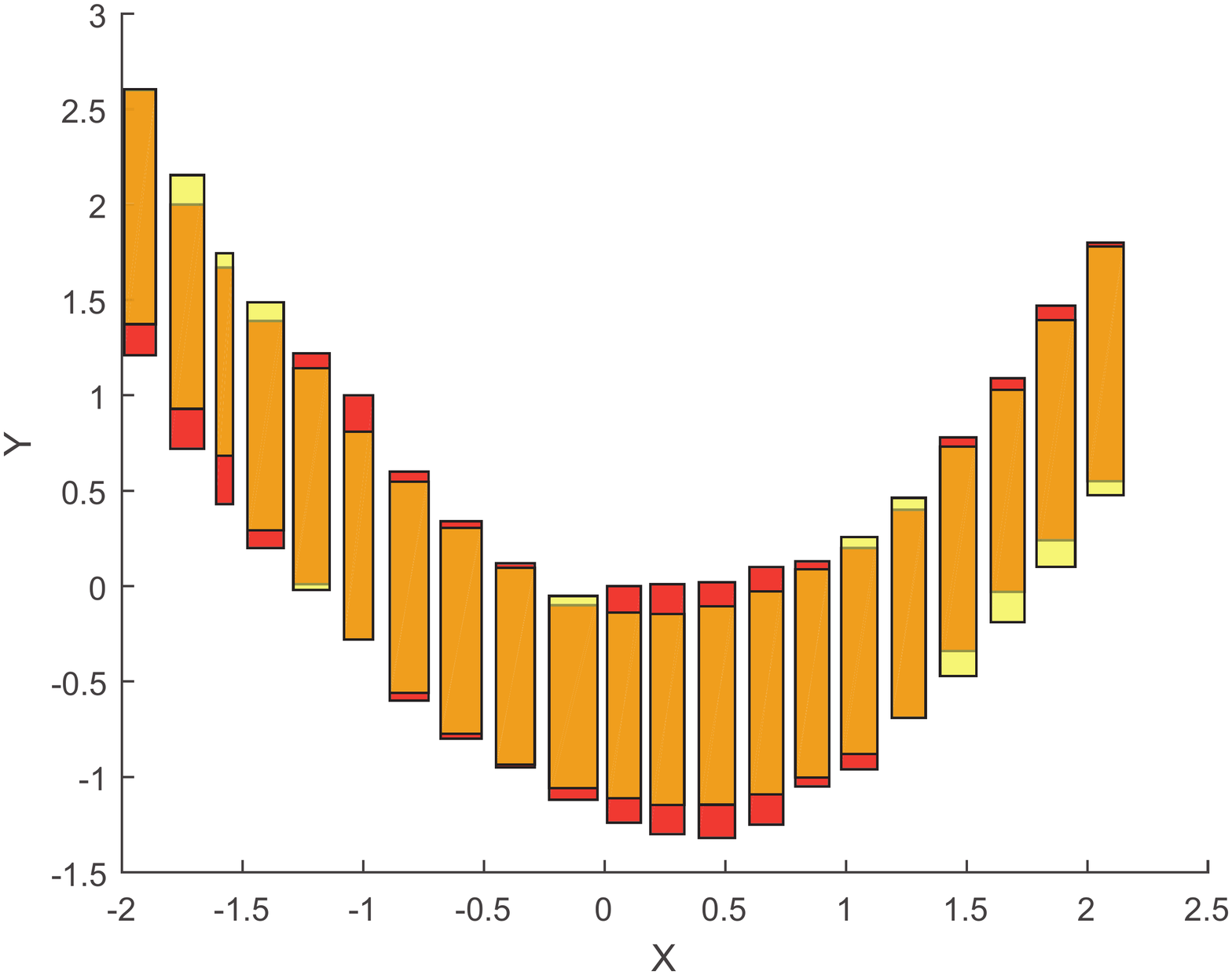}}\label{fig_ploy_fit_1a}
    \end{subfigure}
    \hfill
    \begin{subfigure}[For $w=0.5$]
    {\includegraphics[scale=0.26]{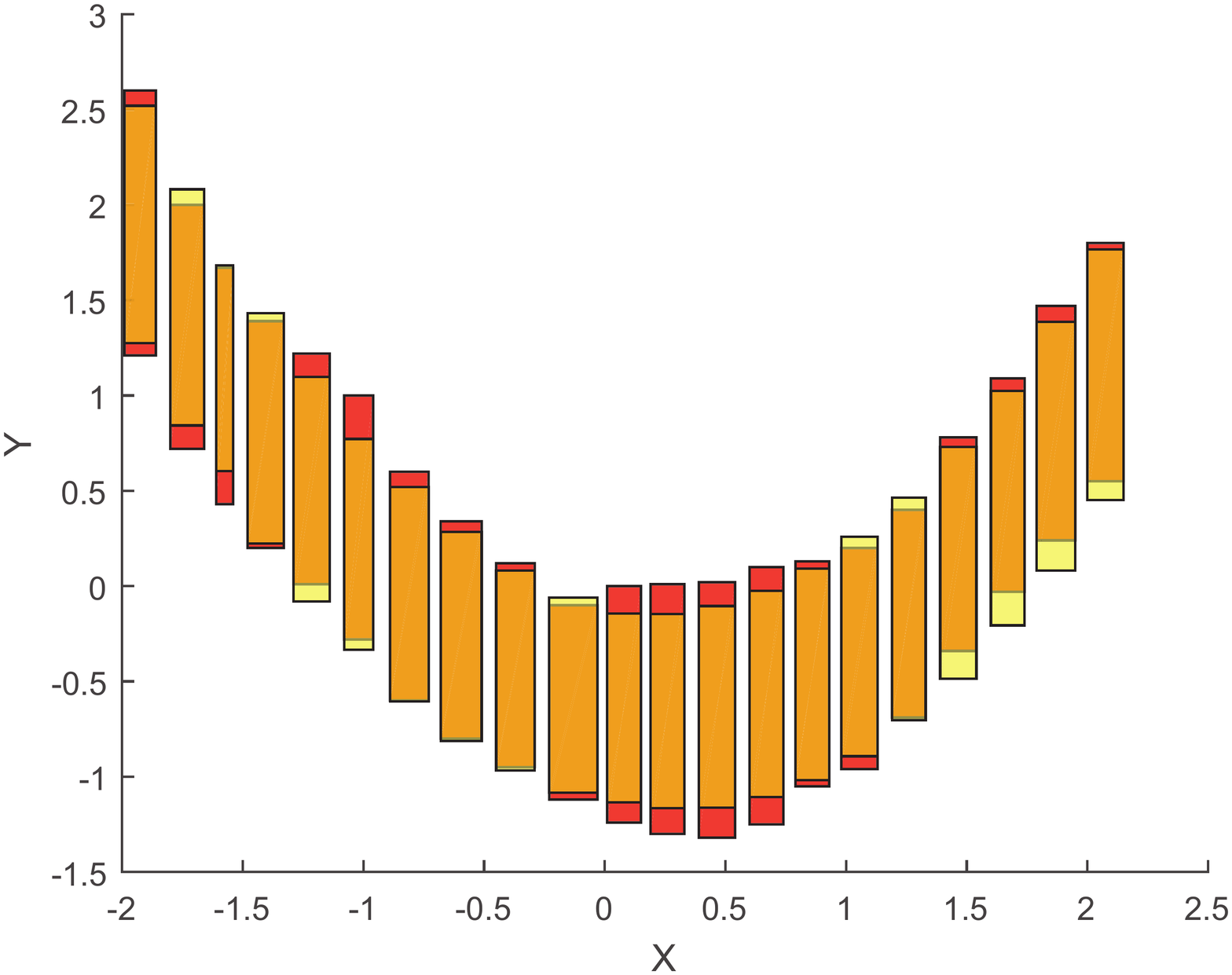}}\label{fig_ploy_fit_1b}
    \end{subfigure}
    \caption{Comparison of the actual $\textbf{Y}_k$ with the estimated output $\boldsymbol{\mathcal{H}}^1\left(\textbf{X}_k;\hat{\beta}\right)$, by the polynomial fitting, of the interval-valued data $\textbf{X}_k$ in Table \ref{tabledata-1}, $k = 1, 2, \ldots, 21$}\label{fpf}
\end{center}
\end{figure}

\subsection{Logistic Fitting}

\noindent Let us consider a set of $15$ pairs of interval-valued data that are displayed in Table \ref{tabledata-2} and we fit a logistic function $\boldsymbol{\mathcal{H}}^2\left(\cdot ;\beta\right):I(\mathbb{R})\to I(\mathbb{R})$ defined by
    \[
    \boldsymbol{\mathcal{H}}^2\left(\textbf{X} ;\beta\right)=\frac{\textbf{1}}{\textbf{1}\oplus e^{-(\beta_1\odot \textbf{C} \oplus \beta_2 \odot \textbf{X})}},
    \]
    where $\textbf{C}$ is a constant interval and $\beta=(\beta_1, \beta_2) \in \mathbb{R}^2$. Thus, the partial $gH$-derivative s of $\boldsymbol{\mathcal{H}}^2\left(\textbf{X}_k; \cdot \right)$ with respect to  $\beta_1$ and $\beta_2$ are
\[
D_1 \boldsymbol{\mathcal{H}}^2\left(\textbf{X}_k; \beta \right)= \frac{\textbf{1}}{\left[\textbf{1}\oplus e^{-(\beta_1\odot \textbf{C} \oplus \beta_2 \odot \textbf{X}_k)}\right]^2}\odot \textbf{C}
~\text{ and }~ D_2 \boldsymbol{\mathcal{H}}\left(\textbf{X}_k; \beta \right)= \frac{\textbf{1}}{\left[\textbf{1}\oplus e^{-(\beta_1\odot \textbf{C} \oplus \beta_2 \odot \textbf{X}_k)}\right]^2}\odot \textbf{X}_k,
\]
respectively, for all $k= 1, 2, \ldots, 15$.
\begin{table}[H]
\caption{Data for logistic fitting}\label{tabledata-2}
\begin{center}
\begin{tabular}{llll}
\toprule
$\textbf{X}_k$ & $\textbf{Y}_k$ & $\textbf{X}_k$ & $\textbf{Y}_k$\\
\midrule
$[-2.70, -2.55]$ & $[10^{-8}, 13 \times 10^{-5}]$ & $[-2.51, -2.35]$ & $[6 \times 10^{-8}, 10^{-3}]$ \\

$[-2.32, -2.23]$ & $[3\times 10^{-7}, 0.003]$ & $[-2.19, -2.02]$ & $[10^{-6}, 0.02]$ \\

$[-2.00, -1.83]$ & $[6\times 10^{-6}, 0.078]$ & $[-1.79, -1.65]$ & $[3\times 10^{-5}, 0.301]$ \\

$[-1.60, -1.42]$ & $[25\times10^{-5}, 0.760]$ & $[-1.39, -1.20]$ & $[15\times10^{-4}, 0.930]$ \\

$[-1.16, -0.98]$ & $[0.012, 0.980]$ & $[-0.94, -0.72]$ & $[0.080, 0.990]$ \\

$[-0.69, -0.50]$ & $[0.440, 0.998]$ & $[-0.46, -0.28]$ & $[0.880, 0.999]$ \\

$[-0.24, -0.02]$ & $[0.981, 1.000]$ & $[0.02, 0.21]$ & $[0.997, 0.999]$ \\

$[0.26, 0.44]$ & $[0.998, 1.000]$ \\
\bottomrule
\end{tabular}
\end{center}
\end{table}

Taking $\textbf{C} = [1.30, 3.40]$ and the initial value of $\beta$ as $(7, -4)^T$, and applying Algorithm \ref{algoes2} with $w=0.7$ on the IOP (\ref{eef}) corresponding to the function $\boldsymbol{\mathcal{H}}^2\left(\cdot ;\beta\right)$, we obtain the value of $\hat{\beta}$ as $(3.3940,  8.5835)^T$ in $10$ iterations.\\

Figure \ref{flf} shows the comparison of the actual interval-valued output $\textbf{Y}_k$ with the estimated output $\boldsymbol{\mathcal{H}}^2\left(\textbf{X}_k ;\hat{\beta}\right)$ of the interval-valued data $\textbf{X}_k$ in logistic fitting for the value of $w = 0.7$, $k = 1, 2, \ldots, 12$. The common portions of $\textbf{Y}_k$ with $\boldsymbol{\mathcal{H}}^2\left(\textbf{X}_k ;\hat{\beta}\right)$ are illustrated by orange regions. The extended portions of $\textbf{Y}_k$ and $\boldsymbol{\mathcal{H}}^2\left(\textbf{X}_k ;\hat{\beta}\right)$ are depicted by red and yellow regions, respectively.

\begin{figure}[H]
\begin{center}
\includegraphics[scale=0.35]{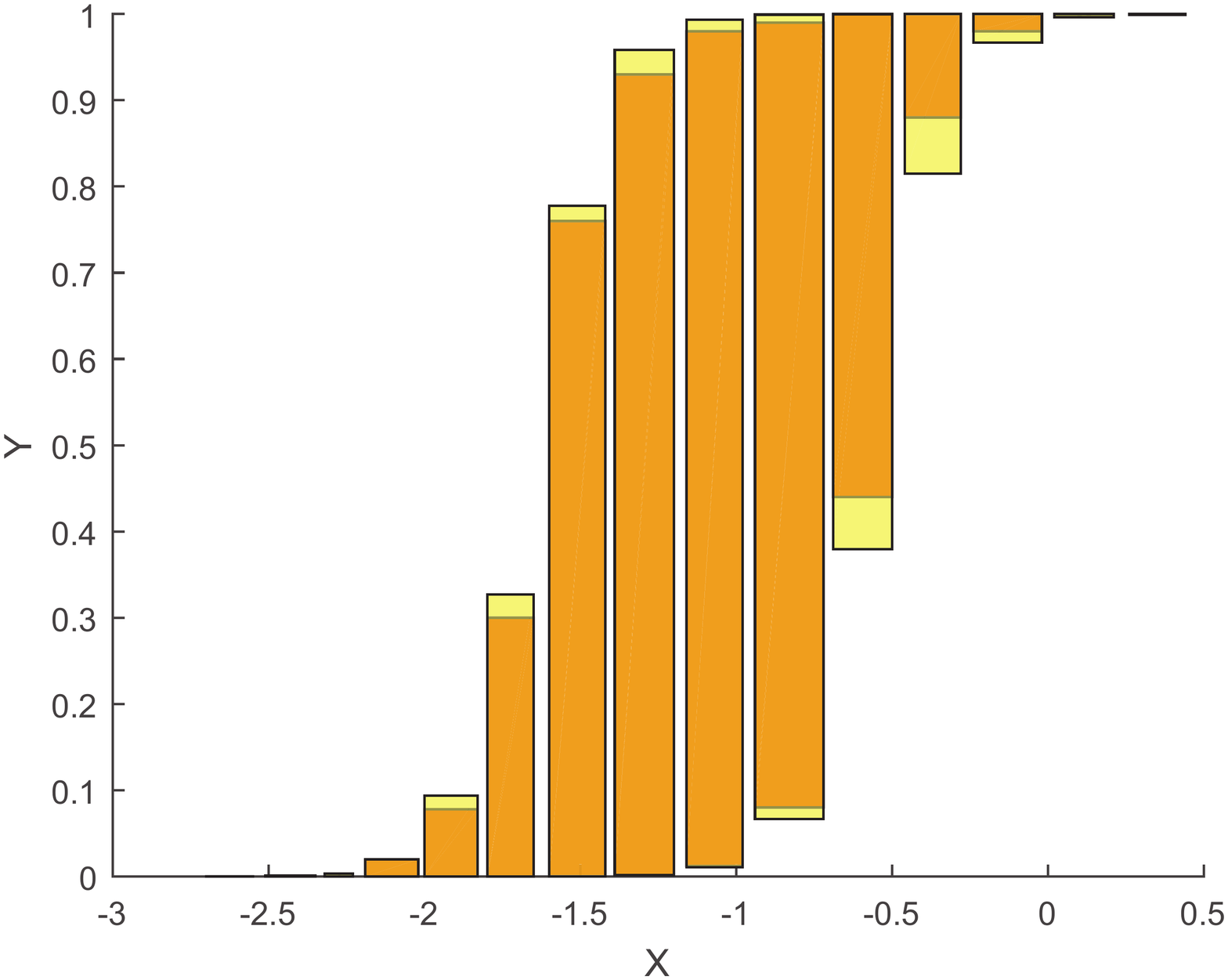}
  \caption{Comparison of the actual $\textbf{Y}_k$ with the estimated output $\boldsymbol{\mathcal{H}}^2\left(\textbf{X}_k ;\beta\right)$ of the interval-valued data $\textbf{X}_k$ in logistic fitting for the data in Table \ref{tabledata-2}, $k = 1, 2, \ldots, 15$}\label{flf}
\end{center}
\end{figure}

\section{Conclusion and Future Directions} \label{scf}

\noindent In this article, a general $gH$-gradient efficient-direction method and a $\mathcal{W}$-gradient efficient method for IOPs have been developed. The convergence analysis and the algorithmic implementations of both the methods have been presented. It has been shown that the $\mathcal{W}$-gradient efficient method converges linearly for a strongly convex interval-valued objective function. In the sequel, a few topics of calculus and convexity of IVFs have been proposed which were needed for the development of the methods. For a $gH$-differentiable IVFs, the relation between convexity and the gradient of a $gH$-differentiable IVF and an optimality condition of an IOP have been derived. Further, a notion of efficient-direction for IVFs has been introduced which is used to develop the general $gH$-gradient efficient and $\mathcal{W}$-gradient efficient methods. The proposed $\mathcal{W}$-gradient efficient method has been used to solve the least square problems with interval-valued data. The application has been exemplified by a polynomial fitting and a logistic curve fitting. \\

{\color{red}In connection with the proposed research, future research can evolve in several directions as follows.}
\begin{itemize}
    \item One may attempt to develop an appropriate approach to find the step lengths $\alpha_k$'s in Algorithms \ref{algoes1} and \ref{algoes2}. Towards this, one can also concentrate on developing exact or inexact line search techniques for IOPs.

    \item It is notable that in the definition of strongly convex IVFs (Definition \ref{strongly_convex_IVF}), we have taken a degenerate interval $[\sigma, \sigma]$ for some $\sigma > 0$. One can attempt to make a generalization of the used  $[\sigma, \sigma]$ to $[\sigma_1,\sigma_2]$.

    \item Analysis of the proposed method for more flexible IVFs, especially for nonconvex  IVFs can be performed in the future.

    \item One may attempt to apply the proposed logistic regression on the classification problems with interval-valued data.

    \item In the future, a rigorous error analysis of the red and yellow portions of Figure \ref{fpf} and Figure \ref{flf} can be performed.

    \item Future research can be made on applying the proposed methods in constrained least square problems with interval-valued data.

    \item Applications of least square technique in practical problems can be performed.
    
    \item {\color{red} Also in the future, one can try to develop the proposed methods of this article based on constrained interval analysis \cite{lodwick2018}}.\\
    
\end{itemize}

%
%
%

\appendix

\section{Proof of Lemma \ref{ldr1}} \label{apdr}
\begin{proof}
\begin{enumerate}[(i)]
\item \begingroup\allowdisplaybreaks\begin{align*}
\textbf{A} ~\preceq~ \textbf{B} & \Longleftrightarrow [\underline{a}, \overline{a}] ~\preceq~ [\underline{b}, \overline{b} ]\\
& \Longleftrightarrow \underline{a} ~\leq~ \underline{b} ~\text{and}~ \overline{a} ~\leq~ \overline{b}\\
& \Longleftrightarrow \underline{a}-\underline{b} ~\leq~ 0 ~\text{and}~ \overline{a} -\overline{b} ~\leq~ 0\\
& \Longleftrightarrow \left[\min\left\{\underline{a}-\underline{b}, \overline{a} -\overline{b}\right\},~\max\left\{\underline{a}-\underline{b}, \overline{a} -\overline{b}\right\}\right] ~\preceq~ \textbf{0} \\
& \Longleftrightarrow \textbf{A}\ominus_{gH}\textbf{B} ~\preceq~ \textbf{0}.
\end{align*}\endgroup
\item According to the Definition \ref{dghd}, $\textbf{A}=\textbf{B}\Longleftrightarrow\textbf{A}\ominus_{gH}\textbf{B}=\textbf{0}$.\\
If $\textbf{A} ~\neq~ \textbf{B}$, then
\begingroup\allowdisplaybreaks\begin{align*}
\textbf{A}~\nprec~ \textbf{B} & \Longleftrightarrow [\underline{a}, \overline{a}] ~\nprec~ [\underline{b}, \overline{b} ]\\
& \Longleftrightarrow \text{at least one of the inequalities}~~\underline{a} ~>~ \underline{b}~~\text{and}~~\overline{a} ~>~ \overline{b}~~\text{is true}\\
& \Longleftrightarrow \max\left\{\underline{a}-\underline{b}, \overline{a} -\overline{b}\right\} ~>~ 0 \\
& \Longleftrightarrow \textbf{A}\ominus_{gH}\textbf{B} ~\nprec~ \textbf{0}.
\end{align*}\endgroup
\end{enumerate}
\end{proof}
\section{Proof of Lemma \ref{lc1}} \label{aplc1}
\begin{proof}
Let $\textbf{F}$ be $gH$-continuous at a point $\bar{x}$ of the set $\mathcal{X}$. Thus, for any $d\in\mathbb{R}^n$ such that $\bar{x}+d\in\mathcal{X}$,
\[
\lim_{\lVert d \rVert\to 0}\left(\textbf{F}(\bar{x}+d)\ominus_{gH}\textbf{F}(\bar{x})\right)=\textbf{0},
\]
which implies
\[
\lim_{\lVert d \rVert\to 0}\left([\underline{f}(\bar{x}+d),\ \overline{f}(\bar{x}+d)]\ominus_{gH}[\underline{f}(\bar{x}),\ \overline{f}(\bar{x})]\right)=[0,\ 0].
\]
Hence, by the definition of $gH$-difference we have
\[
\lim_{\lVert d \rVert\to 0}(\underline{f}(\bar{x}+d)-\underline{f}(\bar{x}))\to 0\ \mbox{and}\ \lim_{\lVert d \rVert\to 0}(\overline{f}(\bar{x}+d)-\overline{f}(\bar{x}))\to 0,
\]
i.e., $\underline{f}$ and $\overline{f}$ are continuous at $\bar{x}\in\mathcal{X}$.\\

Conversely, let the functions $\underline{f}$ and $\overline{f}$ be continuous at $\bar{x}\in\mathcal{X}$. If possible, let $\textbf{F}$ be not $gH$-continuous at $\bar{x}$. Then, as $\lVert d \rVert\to 0,\ (\textbf{F}(\bar{x}+d)\ominus_{gH}\textbf{F}(\bar{x}))\not\to\textbf{0}$. Therefore, as $\lVert d \rVert\to 0$ at least one of the functions $(\underline{f}(\bar{x}+d)-\underline{f}(\bar{x}))$ and $(\overline{f}(\bar{x}+d)-\overline{f}(\bar{x}))$ does not tend to $0$. So it is clear that at least one of the functions $\underline{f}$ and $\overline{f}$ is not continuous at $\bar{x}$. This contradicts the assumption that the functions $\underline{f}$ and $\overline{f}$ both are continuous at $\bar{x}$. Hence, $\textbf{F}$ is $gH$-continuous at $\bar{x}$. \\
\end{proof}
\section{Proof of Lemma \ref{lsc}} \label{aplsc}
\begin{proof}
The result is followed by Lemma \ref{lc1} and the fact that the real-valued functions $\underline{f}$ and $\overline{f}$ are continuous at the point $\bar{x}\in\mathcal{X}$ if and only if for every sequence $\{x_n\}$ in $\mathcal{X}$ converging to $\bar{x}$, the sequences $\{\underline{f}(x_n)\}$ and $\{\overline{f}(x_n)\}$ converge to $\underline{f}(\bar{x})$ and $\overline{f}(\bar{x})$, respectively. \\
\end{proof}
\section{Proof for the problem in Remark \ref{exlivf2}} \label{aplivf}
\begin{proof}
Due to Definition \ref{dlivf},
\[
\textbf{F}(y)=\bigoplus_{i=1}^n y_i\odot\textbf{F}(e_i)~ \text{for all}~y=(y_1, y_2,\ldots,y_n)^T\in \mathcal{X}.
\]
Replacing $y$ by $\lambda x$, where $\lambda \in \mathbb{R}$ and $x=(x_1, x_2,\ldots,x_n)^T\in \mathcal{X}$, we have
\begingroup\allowdisplaybreaks\begin{align*}
\textbf{F}(\lambda x) = \bigoplus_{i=1}^n \lambda x_i\odot\textbf{F}(e_i) = \lambda \odot \bigoplus_{i=1}^n x_i\odot\textbf{F}(e_i) = \lambda\odot\textbf{F}(x) \end{align*}
\endgroup

\noindent Let us consider a pair of elements $x=(x_1, x_2,\ldots, x_n)^T$ and $y=(y_1, y_2,\ldots,y_n)^T$ of $\mathcal{X}$. If the corresponding $x_i$ and $y_i$ are of same sign for all $i\in\{1,\ 2,\ \ldots,\ n\}$, then
\begingroup\allowdisplaybreaks\begin{align*}
\textbf{F}(x)\oplus\textbf{F}(y)&=\bigoplus_{i=1}^n x_i\odot\textbf{F}(e_i)\oplus\bigoplus_{i=1}^n y_i\odot\textbf{F}(e_i)\\
&=\bigoplus_{i=1}^n (x_i+y_i)\odot\textbf{F}(e_i), ~~\text{by Remark \ref{ria1}}\\
&=\textbf{F}(x+y).
\end{align*}\endgroup
Let a few of the $x_i$'s and the corresponding $y_i$'s are of different signs. Without loss of generality, let the first $p$ number of $x_i$'s are of same signs with their corresponding $y_i$'s, thereafter consequtive $q-p$ numbers of $x_i$'s are nonnegative but corresponding $y_i$'s are nonpositive, and the last $n-q$ numbers of $x_i$'s are nonpositive but corresponding $y_i$'s are nonnegative, where $p~\leq~ q ~\leq~ n$. Also, let
\[
\textbf{F}(e_i)=\textbf{A}_i=[\underline{a}_i, \overline{a}_i],~\text{for all}~e_i,~ i \in \{1, 2, \ldots, n\}
\]
Then, we have
\begingroup\allowdisplaybreaks\begin{align*}
&~\textbf{F}(x)\oplus\textbf{F}(y) ~=~\left(\bigoplus_{k=1}^p x_k\odot\textbf{F}(e_k) \oplus\bigoplus_{l=p+1}^q x_l\odot\textbf{F}(e_l) \oplus\bigoplus_{m=q+1}^n x_m\odot\textbf{F}(e_m)\right)\\
&~~~~~~~~~~~~~~~~~~~~~~~~~~~\oplus\left(\bigoplus_{k=1}^p y_k\odot\textbf{F}(e_k) \oplus\bigoplus_{l=p+1}^q y_l\odot\textbf{F}(e_l) \oplus\bigoplus_{m=q+1}^n y_m\odot\textbf{F}(e_m)\right)\\
=~&\bigoplus_{k=1}^p (x_k+y_k)\odot\textbf{F}(e_k) \oplus\bigoplus_{l=p+1}^q \left(x_l\odot\textbf{F}(e_l) \oplus y_l\odot\textbf{F}(e_l)\right)\\
&~~~~~~~~~~~~~~~~~~~~~~~~~~~\oplus\bigoplus_{m=q+1}^n \left(x_m\odot\textbf{F}(e_m) \oplus y_m\odot\textbf{F}(e_m)\right)\\
=~&\textbf{B} \oplus\bigoplus_{l=p+1}^q \left(x_l\odot\textbf{F}(e_l) \ominus v_l\odot\textbf{F}(e_l)\right)\oplus\bigoplus_{m=q+1}^n \left(y_m\odot\textbf{F}(e_m) \ominus u_m\odot\textbf{F}(e_m)\right),\\
&~~~~~~~~~~~~~~~~~~~~~\text{where}~\textbf{B}=\bigoplus_{k=1}^p (x_k+y_k)\odot\textbf{F}(e_k),~v_l=-y_l~\text{and}~u_m=-x_m\\
&~~~~~~~~~~~~~~~~~~~~~~~~~~~~~~~~~~~~\text{for all}~l=p+1,~p+2,~\ldots,~q~\text{and}~m=q+1,~q+2,~\ldots,~r \\
=~&\textbf{B}\oplus\Bigg[\sum_{l=p+1}^q(\underline{a}_lx_l-\overline{a}_lv_l) +\sum_{m=q+1}^n(\underline{a}_my_m-\overline{a}_mu_m), \\
&~~~~~~~~~~~~~~~~~~~~~\sum_{l=p+1}^q(\overline{a}_lx_l-\underline{a}_lv_l) +\sum_{m=q+1}^n(\overline{a}_my_m-\underline{a}_mu_m) \Bigg]\\
=~&\textbf{B}\oplus\textbf{C},~~~\text{where}~\underline{c}=\sum_{l=p+1}^q(\underline{a}_lx_l-\overline{a}_lv_l) +\sum_{m=q+1}^n(\underline{a}_my_m-\overline{a}_mu_m)\\
&~~~~~~~~~~~~~~~~~~~~\text{and}~\overline{c}=\sum_{l=p+1}^q(\overline{a}_lx_l-\underline{a}_lv_l) +\sum_{m=q+1}^n(\overline{a}_my_m-\underline{a}_mu_m).
\end{align*}\endgroup
Further, we note that
\begingroup\allowdisplaybreaks\begin{align*}
\textbf{F}(x+y)
=~&\bigoplus_{k=1}^p (x_k+y_k)\odot\textbf{F}(e_k) \oplus\bigoplus_{l=p+1}^q \left(x_l+y_l\right)\odot\textbf{F}(e_l)\oplus\bigoplus_{m=q+1}^n \left(x_m+y_m\right)\odot\textbf{F}(e_m) \\
=~&\textbf{B} \oplus\bigoplus_{l=p+1}^q \left(x_l-v_l\right)\odot\textbf{F}(e_l)\oplus\bigoplus_{m=q+1}^n \left(y_m-u_m\right)\odot\textbf{F}(e_m)\\
=~&\textbf{B}\oplus\bigoplus_{l=p+1}^q\left(x_l-v_l\right)\odot\left[\underline{a}_l, \overline{a}_l\right]\oplus\bigoplus_{m=q+1}^n \left(y_m-u_m\right)\odot\left[\underline{a}_m, \overline{a}_m\right]
\end{align*}\endgroup
Again, without loss of generality, we let that among $q-p$ numbers of $x_l-v_l$ the first $r-p$ elements are nonnegative and rest are nonpositive, where $p~\leq~ r~\leq~ q$. Similarly, also we let that among $n-q$ numbers of $y_m-u_m$ the first $s-q$ elements are nonnegative and rest are nonpositive, where $q~\leq~ s~\leq~ n$. Then, we have
$
\textbf{F}(x+y)=\textbf{B}\oplus\textbf{D},
$
where
\[
\underline{d}=\sum_{l=p+1}^r\underline{a}_l\left(x_l-v_l\right) +\sum_{l=r+1}^q\overline{a}_l\left(x_l-v_l\right) +\sum_{m=q+1}^s\underline{a}_l\left(y_m-u_m\right) +\sum_{m=s+1}^n\overline{a}_l\left(y_m-u_m\right)
\]
and
\[
\overline{d}=\sum_{l=p+1}^r\overline{a}_l\left(x_l-v_l\right) +\sum_{l=r+1}^q\underline{a}_l\left(x_l-v_l\right) +\sum_{m=q+1}^s\underline{a}_l\left(y_m-u_m\right) +\sum_{m=s+1}^n\overline{a}_l\left(y_m-u_m\right).
\]
Since all $x_l$, $v_l$, $y_m$, $u_m$ are positive for all $p+1~\leq~ l ~\leq~ q$ and $q+1 ~\leq~ m ~\leq~ n$, we obtain
\begingroup\allowdisplaybreaks\begin{align*}
&\underline{a}_lx_l-\overline{a}_lx_l~\leq~\underline{a}_l(x_l-v_l), ~~\underline{a}_lx_l-\overline{a}_lx_l~\leq~\overline{a}_l(x_l-v_l),\\
&\underline{a}_my_m-\overline{a}_mu_m~\leq~\underline{a}_m(y_m-u_m), ~~\underline{a}_my_m-\overline{a}_mu_m~\leq~\overline{a}_m(y_m-u_m),\\
&\overline{a}_lx_l-\underline{a}_lx_l~\geq~\overline{a}_l(x_l-v_l), ~~\overline{a}_lx_l-\underline{a}_lx_l~\geq~\underline{a}_l(x_l-v_l),\\
&\overline{a}_my_m-\underline{a}_mu_m~\geq~\overline{a}_m(y_m-u_m),~~\text{and}~~\overline{a}_my_m-\underline{a}_mu_m~\geq~\underline{a}_m(y_m-u_m).
\end{align*}\endgroup
Thus, we get
\begingroup\allowdisplaybreaks\begin{align*}
&\underline{c}~\leq~ \underline{d}~~\text{and}~~\overline{c}~\geq~\overline{d},\\
&\underline{b}+\underline{c}~\leq~ \underline{b}+\underline{d}~~\text{and}~~\overline{b}+\overline{c} ~\geq~\overline{b}+\overline{d}.
\end{align*}\endgroup
Therefore, either $\textbf{F}(x)\oplus\textbf{F}(y)=\textbf{B}\oplus\textbf{C}$ and $\textbf{F}(x+y)=\textbf{B}\oplus\textbf{D}$ are equal or none of them dominates the other for all $x$ and $y$ in $\mathcal{X}$.
\end{proof}
\section{Proof of Lemma \ref{ld1}} \label{apld1}
\begin{proof}
Let $\textbf{F}$ be $gH$-differentiable at $\bar{x}\in\mathcal{X}$. By Definition \ref{dghd}, there exists a $\delta ~>~ 0$ such that
\begin{equation}\label{edf1}
\left(\textbf{F}(\bar{x}+d) \ominus_{gH} \textbf{F}(\bar{x})\right) \ominus_{gH} \textbf{L}_{\bar{x}}(d)=\lVert d \rVert \odot \textbf{E}(\textbf{F}(\bar{x}); d) ~~\text{for all}~~\lVert d \rVert ~<~ \delta,
\end{equation}
where $\textbf{E}(\textbf{F}(\bar{x}); d) \rightarrow \textbf{0}$ as $\lVert d \rVert \rightarrow 0$.\\

\noindent Considering $d = \lambda h$ for $\lambda ~\neq~ 0$ and $h \in \mathbb{R}^n$ with $|\lambda|\lVert h \rVert ~<~ \delta$, from the equation (\ref{edf1}), we obtain
\begingroup\allowdisplaybreaks\begin{align*}
&\tfrac{1}{\lambda} \odot \left[\left(\textbf{F}(\bar{x} + \lambda h) \ominus_{gH} \textbf{F}(\bar{x})\right) \ominus_{gH} \textbf{L}_{\bar{x}}(\lambda h)\right] = \tfrac{|\lambda|\lVert d \rVert}{\lambda} \odot \textbf{E}(\textbf{F}(\bar{x}); h) \\
\text{or, } &\tfrac{1}{\lambda} \odot \left(\textbf{F}(\bar{x} + \lambda h) \ominus_{gH} \textbf{F}(\bar{x})\right) \ominus_{gH} \tfrac{1}{\lambda} \odot \textbf{L}_{\bar{x}}(\lambda h) = \tfrac{|\lambda|\lVert d \rVert}{\lambda} \odot \textbf{E}(\textbf{F}(\bar{x}); h) \\
\text{or, } &\lim_{\lambda \to 0} \tfrac{1}{\lambda} \odot \left(\textbf{F}(\bar{x} + \lambda h) \ominus_{gH} \textbf{F}(\bar{x})\right)
\ominus_{gH} \textbf{L}_{\bar{x}}(h) = \textbf{0}, ~\text{since}~ \textbf{L}_{\bar{x}}(\lambda h) = \lambda \odot \textbf{L}_{\bar{x}}(h).
\end{align*}\endgroup
Hence,
\begin{equation}\label{edf2}
\lim_{\lambda \to 0} \tfrac{1}{\lambda} \odot \left(\textbf{F}(\bar{x} + \lambda h) \ominus_{gH} \textbf{F}(\bar{x})\right)
=\textbf{L}_{\bar{x}}(h).
\end{equation}
\end{proof}
\section{Proof of Theorem \ref{td1}} \label{aptd1}
\begin{proof}
If $d=(0, 0, \ldots, 0)^T \in \mathbb{R}^n$, both the sides of the equation (\ref{edf3}) become $\textbf{0}$. Hence, the equation (\ref{edf3}) is trivially true for $d=(0, 0, \ldots, 0)^T$.\\

\noindent Let us assume that $d ~\neq~ (0, 0, \ldots, 0)^T$. Since $\textbf{F}$ is $gH$-differentiable at $\bar{x}$, by Lemma \ref{ld1} there exists a $\delta~>~0$ such that
\begin{equation}\label{edf4}
\lim_{\lambda \to 0} \tfrac{1}{\lambda} \odot \left(\textbf{F}(\bar{x} + \lambda h) \ominus_{gH} \textbf{F}(\bar{x})\right)
=\textbf{L}_{\bar{x}}(h),
\end{equation}
where $d = \lambda h$ with $\lambda ~\neq~ 0$, $h \in \mathbb{R}^n$ and $|\lambda|\lVert h \rVert ~<~ \delta$.\\

\noindent Taking $h=e_i$, the $i$-th unit vector in the standard basis of $\mathbb{R}^n$, from the equation (\ref{edf4}) we obtain
\begingroup\allowdisplaybreaks\begin{align*}
&\lim_{\lambda \rightarrow 0} \tfrac{1}{\lambda} \odot \left(\textbf{F}(\bar{x} + \lambda e_i) \ominus_{gH} \textbf{F}(\bar{x})\right)
= \textbf{L}_{\bar{x}}(e_i)\\
\text{or, } & D_i\textbf{F}(\bar{x})= \textbf{L}_{\bar{x}}(e_i).
\end{align*}\endgroup
Therefore, all the $i$-th  partial $gH$-derivative  $D_i \textbf{F}(\bar{x})$ of $\textbf{F}$ at $\bar{x}$ exist. Hence, the gradient of $\textbf{F}$ at $\bar{x}$ exists.\\

\noindent According to Definition \ref{dlivf} of linear IVF, we get
\begingroup\allowdisplaybreaks\begin{align*}
\textbf{L}_{\bar{x}}(d) & = \textbf{L}_{\bar{x}}(d_1, d_2, \ldots, d_n) \\
                    & = d_1 \odot \textbf{L}_{\bar{x}}(e_1) \oplus d_2 \odot \textbf{L}_{\bar{x}}(e_2) \oplus \cdots \oplus d_n \odot \textbf{L}_{\bar{x}}(e_n) \\
                    & = \bigoplus_{i=1}^n d_i \odot D_i\textbf{F}(\bar{x}) = d^T \odot \nabla\textbf{F}(\bar{x}).
\end{align*}\endgroup
\end{proof}
\section{Proof of Theorem \ref{td1a}} \label{aptd1a}
\begin{proof}
Consider an arbitrary $\bar{x}\in \mathcal{X}$. Let for any $d\in \mathcal{N}_\delta(\bar{x})\cap \mathcal{X}$,
\[
    \textbf{F}(\bar{x}+d)=\textbf{F}(\bar{x})\oplus\textbf{F}(d).
\]
Since
\[\nabla\textbf{F}(\bar{x})=\left(\textbf{F}(e_1), \textbf{F}(e_2), \ldots, \textbf{F}(e_n)\right)^T,
\]
for any $d\in \mathcal{N}_\delta(\bar{x})\cap \mathcal{X}$, we obtain
\[
d^T\odot \nabla\textbf{F}(\bar{x})=\bigoplus_{i=1}^n d_i \odot \textbf{F}(e_i)=\textbf{F}(d).
\]
Therefore,
\begingroup\allowdisplaybreaks\begin{align*}
&\lim_{\lVert d \rVert \to 0} \tfrac{1}{\lVert d \rVert} \odot \big(\left(\textbf{F}(\bar{x} + d) \ominus_{gH} \textbf{F}(\bar{x})\right)\ominus_{gH}d^T\odot \nabla\textbf{F}(\bar{x})\big)\\
=&
\lim_{\lVert d \rVert \to 0} \tfrac{1}{\lVert d \rVert} \odot \big(\left(\textbf{F}(\bar{x}) \oplus \textbf{F}(d) \ominus_{gH} \textbf{F}(\bar{x})\right)\ominus_{gH}\textbf{F}(d)\big)\\
=&
\lim_{\lVert d \rVert \to 0} \tfrac{1}{\lVert d \rVert} \odot \big(\left[\underline{f}(\bar{x})+\underline{f}(d)-\underline{f}(\bar{x}), \overline{f}(\bar{x})+\overline{f}(d)-\overline{f}(\bar{x})\right] \ominus_{gH}\textbf{F}(d)\big)\\
=&
\lim_{\lVert d \rVert \to 0} \tfrac{1}{\lVert d \rVert} \odot \left(\textbf{F}(d)\ominus_{gH}\textbf{F}(d)\right)\\
=& \textbf{0}.
\end{align*}\endgroup
Hence, due to Definition \ref{dghd} and Theorem \ref{td1}, the linear IVF $\textbf{F}$ is differentiable at $\bar{x}\in \mathcal{X}$.
\end{proof}
\section{Proof of Theorem \ref{td2}} \label{aptd2}
\begin{proof}
Let the function $\textbf{F}$ be convex on $\mathcal{X}$. Then, for any $x,~y \in \mathcal{X}$ and $\lambda\in (0,1]$, we get
\[
\textbf{F}(x+\lambda (y-x))=\textbf{F}(\lambda y+\lambda'x)~\preceq~ \lambda\odot\textbf{F}(y)\oplus\lambda'\odot\textbf{F}(x),~~\text{where}~~\lambda'=1-\lambda.
\]
Hence,
\begingroup\allowdisplaybreaks\begin{align*}
\textbf{F}(x+\lambda (y-x))\ominus_{gH}\textbf{F}(x)\preceq&(\lambda\odot\textbf{F}(y)\oplus\lambda'\odot\textbf{F}(x))\ominus_{gH}\textbf{F}(x)\\
=& \big[\lambda \underline{f}(y)+\lambda'\underline{f}(x), \lambda \overline{f}(y)+\lambda'\overline{f}(x)]\ominus_{gH} [\underline{f}(x), \overline{f}(x)\big]\\
	=& \big[\min \{\lambda \underline{f}(y)+\lambda'\underline{f}(x)-\underline{f}(x), \lambda \overline{f}(y)+\lambda'\overline{f}(x)-\overline{f}(x)  \},\\
	&~~\max \{\lambda \underline{f}(y)+\lambda'\underline{f}(x)-\underline{f}(x), \lambda \overline{f}(y)+\lambda'\overline{f}(x)-\overline{f}(x)  \}\big]\\
	=&\big[\min\{\lambda\underline{f}(y)-\lambda \underline{f}(x), \lambda\overline{f}(y)-\lambda \overline{f}(x) \}, \\
	&~~\max\{\lambda\underline{f}(y)-\lambda \underline{f}(x), \lambda\overline{f}(y)-\lambda \overline{f}(x) \} \big] \\
	=&\lambda \odot \big[\min \{\underline{f}(y)-\underline{f}(x), \overline{f}(y)-\overline{f}(x)\},\\
	&~~~~~~~~~\max \{\underline{f}(y)-\underline{f}(x), \overline{f}(y)-\overline{f}(x)\}  \big],~~\text{since}~~\lambda~>~0\\
	=&\lambda\odot(\textbf{F}(y)\ominus_{gH}\textbf{F}(x)),
\end{align*}\endgroup
which implies
\[
\frac{1}{\lambda}\odot(\textbf{F}(x+\lambda (y-x))\ominus_{gH}\textbf{F}(x))~\preceq~ \textbf{F}(y)\ominus_{gH}\textbf{F}(x).
\]
Since $\textbf{F}$ is $gH$-differentiable at $x\in\mathcal{X}$, taking $\lambda\to 0+$, by Theorem \ref{td1}, we have
\[
(y-x)^T \odot \nabla\textbf{F}(x)~\preceq~ \textbf{F}(y)\ominus_{gH}\textbf{F}(x)~~\text{for all}~x,~y\in \mathcal{X}.
\]
\end{proof}
\section{Proof of Theorem \ref{td3}} \label{aptd3}
\begin{proof}
Let the function $\textbf{F}$ be convex on $\mathcal{X}$. As $\textbf{F}$ is also $gH$-differentiable on $\mathcal{X}$, by Theorem \ref{td2}, for all $x$, $y\in\mathcal{X}$ we obtain
\begin{equation}\label{edf5}
(y-x)^T \odot \nabla\textbf{F}(x)~\preceq~ \textbf{F}(y)\ominus_{gH}\textbf{F}(x)
\end{equation}
and
\begin{equation}\label{edf6}
(x-y)^T \odot \nabla\textbf{F}(y)~\preceq~ \textbf{F}(x)\ominus_{gH}\textbf{F}(y).
\end{equation}
For a given pair of points $x$, $y\in\mathcal{X}$, let us suppose
\[
\textbf{F}(x)\ominus_{gH}\textbf{F}(y)=[\underline{a}, \overline{a}],~
 (x-y)^T \odot \nabla\textbf{F}(x)=[\underline{b}, \overline{b}]~\text{and}~(x-y)^T \odot \nabla\textbf{F}(y)=[\underline{c}, \overline{c}].
\]
Thus, from (\ref{edf5}) and (\ref{edf6}), respectively, we have
\begingroup\allowdisplaybreaks\begin{align*}
&[-\overline{b}, -\underline{b}]\preceq[-\overline{a}, -\underline{a}]~\text{and}~[\underline{c}, \overline{c}]\preceq[\underline{a}, \overline{a}]\\
~\Longrightarrow~ &[\underline{a}, \overline{a}]\preceq[\underline{b}, \overline{b}]~~\text{and}~~[-\overline{a}, -\underline{a}]\preceq[-\overline{c}, -\underline{c}]\\
~\Longrightarrow~ & \underline{a}~\leq~\underline{b},~~\overline{a}~\leq~\overline{b},~~-\underline{a}~\leq~-\underline{c}~~\text{and}~~-\overline{a}~\leq~-\overline{c}\\
~\Longrightarrow~ & 0~\leq~\underline{b}-\underline{c}~~\text{and}~~0~\leq~\overline{b}-\overline{c}\\
~\Longrightarrow~ &\textbf{0}\preceq[\min\{\underline{b}-\underline{c}, \overline{b}-\overline{c}\}, \max\{\underline{b}-\underline{c}, \overline{b}-\overline{c}\}]\\
~\Longrightarrow~ &\textbf{0}\preceq(x-y)^T \odot \nabla\textbf{F}(x)\ominus_{gH}(x-y)^T \odot \nabla\textbf{F}(y).
\end{align*}\endgroup
Since $x$ and $y$ are arbitrary,
\[
\textbf{0}\preceq(x-y)^T \odot \nabla\textbf{F}(x)\ominus_{gH}(x-y)^T \odot \nabla\textbf{F}(y)~\text{ for all } x,~y\in \mathcal{X}.
\]
\end{proof}

$\\$
\textbf{Acknowledgement}\\ \\
The authors put a sincere thanks to the anonymous reviewers and editors for their valuable comments to enhance the paper. The first author gratefully acknowledges the financial support through the Early Career Research Award (ECR/2015/000467), Science \& Engineering Research Board, Government of India.  \\

\end{document}